\documentclass[a4paper,reqno,12pt]{amsart}
\usepackage{t1enc}
\usepackage[margin=1in]{geometry}
\usepackage{ifthen}
\usepackage{graphicx}
\usepackage{xcolor}
\usepackage{rotating}
\usepackage[small]{caption}

\newcommand{\ignore}[1]{}

\newcommand{\abs}[1]{\left| #1 \right|}
\newcommand{\expr}[1]{\left( #1 \right)}

\newcommand{\set}[1]{\left\{ #1 \right\}}
\newcommand{\scalar}[1]{\left\langle #1 \right\rangle}
\newcommand{\A}{\mathcal{A}}
\newcommand{\C}{\mathbf{C}}
\newcommand{\R}{\mathbf{R}}

\newcommand{\pr}{\mathbf{P}}

\newcommand{\domain}{\mathcal{D}}

\newcommand{\ind}{\mathbf{1}}
\newcommand{\eps}{\varepsilon}
\newcommand{\ph}{\varphi}
\newcommand{\ro}{\varrho}
\newcommand{\thet}{\vartheta}

\newcommand{\catalan}{\mathcal{C}}

\newcommand{\hl}{{(0, \infty)}}

\newcommand{\laplace}{\mathcal{L}}

\theoremstyle{plain}
\newtheorem{theorem}{Theorem}[section]
\newtheorem{lemma}[theorem]{Lemma}
\newtheorem{corollary}[theorem]{Corollary}
\newtheorem{proposition}[theorem]{Proposition}

\theoremstyle{definition}

\newtheorem{remark}[theorem]{Remark}
\newtheorem{assumption}[theorem]{Assumption}

\theoremstyle{remark}

\numberwithin{equation}{section}

\DeclareMathOperator{\re}{Re}
\DeclareMathOperator{\im}{Im}
\DeclareMathOperator{\var}{Var}
\DeclareMathOperator{\Arg}{Arg}

\newcommand{\formula}[2][nolabel]
{\ifthenelse{\equal{#1}{nolabel}}
 {\begin{align*} #2 \end{align*}}
 {\ifthenelse{\equal{#1}{}}
  {\begin{align} #2 \end{align}}
  {\begin{align} \label{#1} #2 \end{align}}
 }
}

\begin{document}

%
%

\title[First passage times for subordinate Brownian motions]{First passage times for subordinate Brownian motions}

\author{Mateusz Kwa{\'s}nicki, Jacek Ma{\l}ecki, Micha{\l} Ryznar}
\thanks{Mateusz Kwa{\'s}nicki received financial support of the Foundation for Polish Science}
\address{Mateusz Kwa{\'s}nicki, Jacek Ma{\l}ecki, Micha{\l} Ryznar \\ Institute of Mathematics and Computer Science \\ Wroc{\l}aw University of Technology \\ ul. Wybrze{\.z}e Wyspia{\'n}\-skiego 27 \\ 50-370 Wroc{\l}aw, Poland}
\email{mateusz.kwasnicki@pwr.wroc.pl, jacek.malecki@pwr.wroc.pl, michal.{\linebreak}ryznar@pwr.wroc.pl}
\address{Mateusz Kwa{\'s}nicki \\ Institute of Mathematics \\ Polish Academy of Sciences \\ ul. {\'S}niadeckich 8 \\ 00-976 Warszawa, Poland}
\email{m.kwasnicki@impan.pl}
\address{Jacek Ma\l{}ecki \\ LAREMA \\ Universit\'e d'Angers \\ 2 Bd Lavoisier \\ 49045 Angers cedex 1, France}

\begin{abstract}
Let $X_t$ be a subordinate Brownian motion, and suppose that the L{\'e}vy measure of the underlying subordinator has a completely monotone density. Under very mild conditions, we find integral formulae for the tail distribution $\pr(\tau_x > t)$ of first passage times $\tau_x$ through a barrier at $x > 0$, and its derivatives in $t$. As a corollary, we examine the asymptotic behaviour of $\pr(\tau_x > t)$ and its $t$-derivatives, either as $t \to \infty$ or $x \to 0^+$. 
\end{abstract}

\maketitle

%
%

\section{Introduction}
\label{sec:intro}

The present article complements and extends the results of the recent paper~\cite{bib:mk10}, where spectral theory for a class of L{\'e}vy processes killed upon leaving a half-line was developed. In a closely related paper~\cite{bib:kmr11}, first passage times were studied for a rather general class of one-dimensional L\'evy processes. In the present article, more detailed properties of first passage times are established for processes considered in~\cite{bib:mk10}: symmetric L\'evy processes, whose L\'evy measure has a completely monotone density function on $(0, \infty)$. More precisely, we prove asymptotic fomulae, regularity, and estimates of the tail distribution $\pr(\tau_x > t)$ of the first passage time through a barrier at the level $x$ for a L{\'e}vy process $X_t$:
\formula{
 \tau_x & = \inf \set{t \ge 0 : X_t \ge x} , && x \ge 0 ,
}
as well as its derivatives in $t$. Alternatively, the results can be stated in terms of the \emph{supremum functional} $M_t = \sup_{s \in [0, t]} X_s$, since we have $\pr(\tau_x > t) = \pr(M_t < x)$ for all $t, x \ge 0$.

In~\cite{bib:mk10}, a formula was given for generalised eigenfunctions $F_\lambda(x)$ of the transition semigroup of the killed process. As an application, the distribution of first passage times was expressed in terms of the eigenfunctions $F_\lambda(x)$. The full statement of this result was only announced, and a formal proof was given under more restrictive conditions. In the present paper, we provide the proof in the general case (Theorem~\ref{th:fpt}). The expression for the distribution of $\tau_x$ is then used to find estimates and asymptotic expansion of $(d / dt)^n \pr(\tau_x > t)$. This requires detailed analysis of the eigenfunctions $F_\lambda(x)$.

The double Laplace transform (in $t$ and $x$) of $\pr(\tau_x > t)$ is known for general L{\'e}vy processes since 1957 due to the result of Baxter and Donsker (Theorem~1 in~\cite{bib:bd57}). For symmetric L{\'e}vy processes $X_t$ with L{\'e}vy-Khintchin exponent $\Psi(\xi)$,
\formula[eq:bd]{
\begin{aligned}
 & \int_0^\infty \int_0^\infty e^{-\xi x - z t} \pr(\tau_x > t) dx dt \\ & \qquad = \frac{1}{\xi \sqrt{z}} \, \exp \expr{-\frac{1}{\pi} \int_0^\infty \frac{\xi \, \log(z + \Psi(\zeta))}{\xi^2 + \zeta^2} \, d\zeta} .
\end{aligned}
}
However, the double Laplace transform in~\eqref{eq:bd} has been inverted only for few special cases. It is a classical result that for the Brownian motion, $\tau_x$ is the $(1/2)$-stable subordinator. An explicit formula for the distribution of $\tau_x$ was found for the Cauchy process (the symmetric $1$-stable process) by Darling~\cite{bib:d56}, for a compound Poisson process with $\Psi(\xi) = 1 - \cos \xi$ by Baxter and Donsker~\cite{bib:bd57}, and for the Poisson process with drift by Pyke~\cite{bib:p59}. A formula for the single Laplace transform for symmetric L{\'e}vy processes, under some mild assumptions, was given recently in~\cite{bib:kmr11} (see Theorem~\ref{th:suplaplace} below). In the development of the fluctuation theory for L{\'e}vy processes, many new identities involving first passage times were derived (see~\cite{bib:b96, bib:d07, bib:k06, bib:s99} for a general account on fluctuation theory), including various other characterisations of $\pr(\tau_x > t)$, at least in the stable case, see~\cite{bib:bdp08, bib:b73, bib:d56, bib:d87, bib:ds10, bib:gj09, bib:gj10, bib:hk11, bib:hk09, bib:k10a, bib:k10, bib:z64}. First passage times $\tau_x$ and the supremum functional $M_t$ play an important role in many areas of applied probability (\cite{bib:ak05, bib:bnmr01}), mathematical physics (\cite{bib:ka08, bib:kkm07}), and also in potential theory of L\'evy processes (\cite{bib:bbkrsv09, bib:ksv9, bib:ksv10a, bib:ksv10, bib:ksv11}).

The main result of this article is an explicit, applicable for numerical computations expression for $\pr(\tau_x > t)$ for a class of symmetric L{\'e}vy processes, which includes symmetric $\alpha$-stable processes, relativistic $\alpha$-stable processes and geometric $\alpha$-stable processes (in particular, the variance gamma process) and many others. More precisely, the following assumption is in force throughout the article.

\begin{assumption}
\label{asmp}
Any of the following equivalent conditions is satisfied (see Proposition~2.13 in~\cite{bib:mk10}):
\begin{enumerate}
\item[(a)]
{$X_t$} is a \emph{subordinate Brownian motion}, $X_t = B_{Z_t}$, and the L{\'e}vy measure of the subordinator $Z_t$ has a completely monotone density. Here $B_s$ is the one-dimensional Brownian motion ($\var B_s = 2 s$), $Z_t$ is a \emph{subordinator} (nonnegative L{\'e}vy process), and $B_s$ and $Z_t$ are independent processes;
\item[(b)]
$X_t$ is a symmetric L{\'e}vy process, whose L{\'e}vy measure has a completely monotone density on $(0, \infty)$;
\item[(c)]
$X_t$ is a L{\'e}vy process with L{\'e}vy-Khintchine exponent $\Psi(\xi) = \psi(\xi^2)$ for some \emph{complete Bernstein function} $\psi(\xi)$.
\end{enumerate}
We assume that $X_t$ is non-trivial, that is, $X_t$ is not constantly $0$.
\end{assumption}

\begin{remark}
All explicit formulae and estimates proved in this article are given in terms of the complete Bernstein function $\psi(\xi)$. Translation to the L{\'e}vy-Khintchine exponent $\Psi(\xi) = \psi(\xi^2)$ is immediate, but usually results in less elegant expressions.

In this article, the term \emph{explicit formula} is used for an expression involving a finite number of (absolutely convergent) integrals, elementary functions and the function $\psi$. The complete Bernstein function $\psi$ extends to a holomorphic function on $\C \setminus (-\infty, 0]$ (see Preliminaries). Sometimes (namely, in the formula for $F_\lambda(x)$) we also use this holomorphic extension of $\psi$.\qed
\end{remark}

Our proofs are based on the following two theorems.

\begin{theorem}[Corollary~4.2 in~\cite{bib:kmr11}]
\label{th:suplaplace}
We have
\formula[eq:suplaplace]{
\begin{aligned}
 & \int_0^\infty e^{-\xi x} \pr(\tau_x > t) dx = \frac{2}{\pi} \int_0^\infty \frac{\lambda}{\lambda^2 + \xi^2} \, \frac{\psi'(\lambda^2)}{\sqrt{\psi(\lambda^2)}} \times \\ & \hspace*{7em} \times \exp\expr{\frac{1}{\pi} \int_0^\infty \frac{\xi \log \frac{\lambda^2 - \zeta^2}{\psi(\lambda^2) - \psi(\zeta^2)}}{\xi^2 + \zeta^2} \, d\zeta} e^{-t \psi(\lambda^2)} d\lambda
\end{aligned}
}
for all $t, \xi > 0$.
\qed
\end{theorem}

We remark that the above result is proved in~\cite{bib:kmr11} for all symmetric L{\'e}vy processes with L{\'e}vy-Khintchine exponent $\Psi(\xi)$ having strictly positive derivative on $(0, \infty)$.

The transition semigroup $P^\hl_t$ (acting on $L^p(\hl)$ for any $p \in [1, \infty]$) of the process $X_t$ killed upon leaving the half-line $\hl$, and its $L^2(\hl)$ generator $\A_\hl$, are defined formally, for example, in~\cite{bib:mk10}. These notions are only required in the statements of Theorems~\ref{th:eigenfunctions} and~\ref{th:spectral}, and therefore they are not discussed in detail below.

\begin{theorem}[Theorem~1.1 in~\cite{bib:mk10}]
\label{th:eigenfunctions}
For every $\lambda > 0$, there is a bounded continuous function $F_\lambda$ on $\hl$ which is the eigenfunction of $P^\hl_t$, that is,
\formula{
 P^\hl_t F_\lambda(x) & = e^{-t \psi(\lambda^2)} F_\lambda(x)
}
for all $t, x > 0$. The function $F_\lambda$ is characterised by its Laplace transform:
\formula[eq:lf]{
 \laplace F_\lambda(\xi) & = \frac{\lambda}{\lambda^2 + \xi^2} \, \exp\expr{\frac{1}{\pi} \int_0^\infty \frac{\xi}{\xi^2 + \zeta^2} \, \log \frac{\psi'(\lambda^2) (\lambda^2 - \zeta^2)}{\psi(\lambda^2) - \psi(\zeta^2)} \, d\zeta}
}
for $\xi \in \C$ such that $\re \xi > 0$. Furthermore, for $x > 0$ we have
\formula[eq:f]{
 F_\lambda(x) & = \sin(\lambda x + \thet_\lambda) - G_\lambda(x) ,
}
where the \emph{phase shift} $\thet_\lambda$ belongs to $[0, \pi/2)$, and the \emph{correction term} $G_\lambda(x)$ is a bounded, completely monotone function on $(0, \infty)$. More precisely, we have
\formula[eq:theta]{
 \thet_\lambda & = -\frac{1}{\pi} \int_0^\infty \frac{\lambda}{\lambda^2 - \zeta^2} \, \log \frac{\psi'(\lambda^2) (\lambda^2 - \zeta^2)}{\psi(\lambda^2) - \psi(\zeta^2)} \, d\zeta ,
}
and $G_\lambda$ is the Laplace transform of a finite measure $\gamma_\lambda$ on $(0, \infty)$. When $\psi(\xi)$ extends to a function $\psi^+(\xi)$ holomorphic in the upper complex half-plane $\{\xi \in \C : \im \xi > 0\}$ and continuous in $\{\xi \in \C : \im \xi \ge 0$, and furthermore $\psi^+(-\xi) \ne \psi(\lambda)$ for all $\xi > 0$, then the measure $\gamma_\lambda$ is absolutely continuous, and
\formula[eq:gamma]{
\begin{aligned}
 \gamma_\lambda(d\xi) & = \frac{1}{\pi} \expr{\im \frac{\lambda \psi'(\lambda^2)}{\psi(\lambda^2) - \psi^+(-\xi^2)}} \\ & \qquad \times \exp\expr{-\frac{1}{\pi} \int_0^\infty \frac{\xi}{\xi^2 + \zeta^2} \, \log \frac{\psi'(\lambda^2) (\lambda^2 - \zeta^2)}{\psi(\lambda^2) - \psi(\zeta^2)} \, d\zeta} d\xi
\end{aligned}
}
for $\xi > 0$.\qed
\end{theorem}

We introduce the following two conditions:
\formula[eq:fpt:a1]{
 \sup_{\xi > 0} \frac{\xi |\psi''(\xi)|}{\psi'(\xi)} < 2 ,
}
and, given $t_0 > 0$,
\formula[eq:fpt:a2]{
 \int_1^\infty \sqrt{\frac{\psi'(\xi^2)}{\psi(\xi^2)}} \, e^{-t_0 \psi(\xi^2)} d\xi < \infty .
}

\begin{remark}
\begin{enumerate}
\item[(a)]
By Proposition~\ref{prop:cbf:ests}(b), for every complete Bernstein function $\psi$ the supremum in~\eqref{eq:fpt:a1} is not greater than $2$. Condition~\eqref{eq:fpt:a1}, is needed only to assert that $\sup_{\lambda > 0} \thet_\lambda < \pi/2$, and can be replaced by the latter.
\item[(b)]
When $\psi(\xi)$ is \emph{unbounded} (that is, $X_t$ is not a compound Poisson process) and \emph{regularly varying} of order $\ro_0$ at $0$ and $\ro_\infty$ at $\infty$, then $\ro_0, \ro_\infty \in [0, 1]$ and $\lim_{\xi \to 0^+} (\xi |\psi''(\xi)| / \psi'(\xi)) = 1 - \ro_0$, $\lim_{\xi \to \infty} (\xi |\psi''(\xi)| / \psi'(\xi)) = 1 - \ro_\infty$ (see~\cite{bib:bgt87}). Hence~\eqref{eq:fpt:a1} is automatically satisfied.
\item[(c)]
\label{rem:integrability}
Assumption~\eqref{eq:fpt:a2} is a rather mild growth condition on $\psi(\xi)$ for large $\xi$. By Proposition~\ref{prop:cbf:ests}(a), $\sqrt{\psi'(\xi^2) / \psi(\xi^2)} \le 1 / \xi$. Hence, \eqref{eq:fpt:a2} is satisfied for all $n \ge 0$ an all $t_0 > 0$ whenever $\psi(\xi) \ge c \log(\xi)$ ($\xi > 0$) for some $c > 0$. When $\psi(\xi) = f(\log (1 + \xi))$, then~\eqref{eq:fpt:a2} is equivalent to $\int_1^\infty \sqrt{f'(s) / f(s)} \, e^{-t_0 f(s)} ds < \infty$. For example, if $\psi(\xi) = \log(1 + \log(1 + \xi))$, then~\eqref{eq:fpt:a2} holds if and only if $t_0 > 1/2$.\qed
\end{enumerate}
\end{remark}

In particular, many processes frequently found in literature, including symmetric $\alpha$-stable processes, relativistic $\alpha$-stable processes and geometric $\alpha$-stable processes, satisfy~\eqref{eq:fpt:a1} and~\eqref{eq:fpt:a2}. These and some other examples are discussed in Section~\ref{sec:examples}.

The following are the main results of the article. The first of them provides a formula for $\pr(\tau_x > t)$ and its derivatives in $t$, and it was proved in~\cite{bib:mk10} under more restrictive assumptions. The full statement of the theorem (for $n = 0$ and $n = 1$) was announced in~\cite{bib:mk10} as Theorem~1.8.

\begin{theorem}
\label{th:fpt}
If~\eqref{eq:fpt:a1} and~\eqref{eq:fpt:a2} hold for some $t_0 > 0$, then for all $n \ge 0$, $t > t_0$ ($t \ge t_0$ if $n = 0$) and $x > 0$,
\formula[eq:fpt]{
\begin{aligned}
 & (-1)^n \, \frac{d^n}{dt^n} \, \pr(\tau_x > t) \\ & \qquad = \frac{2}{\pi} \int_0^\infty \sqrt{\frac{\psi'(\lambda^2)}{\psi(\lambda^2)}} \, (\psi(\lambda^2))^n e^{-t \psi(\lambda^2)} F_\lambda(x) d\lambda .
\end{aligned}
}
\end{theorem}

In Theorem~4.6 in~\cite{bib:kmr11} it is proved that
\formula[eq:fptest]{
 \frac{1}{20000} \, \min\expr{1, \frac{1}{\sqrt{t \psi(1/x^2)}}} & \le \pr(\tau_x > t) \le 10 \min\expr{1, \frac{1}{\sqrt{t \psi(1/x^2)}}}
}
for \emph{all} $t, x > 0$ for a relatively wide class of symmetric L{\'e}vy processes $X_t$ (similar results for many asymmetric processes are also available in~\cite{bib:kmr11}). Our second result generalises these estimates to derivatives of $\pr(\tau_x > t)$ in $t$, for a restricted class of processes and $t$ large enough. For similar bounds with slightly different assumptions, see Lemma~\ref{lem:fptdensityest}, Corollary~\ref{cor:fptdensityest}, Remark~\ref{rem:fptdensityest} and Proposition~\ref{prop:geomstable}.

\begin{theorem}
\label{th:fptdensityest}
\begin{enumerate}
\item[(a)]
If~\eqref{eq:fpt:a1} and~\eqref{eq:fpt:a2} hold for some $t_0 > 0$, then the distribution of $\tau_x$ is \emph{ultimately completely monotone}, that is, for each fixed $n \ge 0$, $(-1)^n (d^n / dt^n) \pr(\tau_x > t) \ge 0$ for $t$ large enough.
\item[(b)]
If $\ro = \sup_{\xi > 0} (\xi |\psi''(\xi)| / \psi'(\xi)) < 1$ (cf.~\eqref{eq:fpt:a1}), then there are positive constants $c_1(n)$, $c_2(n)$, $c_3(n, \ro)$ such that
\formula[eq:fptpowerest]{
 \frac{c_1(n)}{t^{n + 1/2} \sqrt{\psi(1 / x^2)}} & \le (-1)^n \, \frac{d^n}{dt^n} \, \pr(\tau_x > t) \le \frac{c_2(n)}{t^{n + 1/2} \sqrt{\psi(1 / x^2)}}
}
for $n \ge 0$, $x > 0$ and $t \ge c_3(n, \ro) / \psi(1 / x^2)$. When $c_2(n)$ is replaced by a constant $\tilde{c}_2(n, \ro)$, then the upper bound in~\eqref{eq:fptpowerest} holds for all $t, x > 0$.
\end{enumerate}
\end{theorem}

Next, we study the asymptotic behaviour of $\pr(\tau_x > t)$ as $t \to \infty$ or $x \to 0^+$. The function $V(x)$ is given by an explicit formula; it is the \emph{renewal function of the ascending ladder-height process} (see Preliminaries). The case $n = 0$ has been previously studied in~\cite{bib:gn86}. 

\begin{theorem}
\label{th:fptdensityasymp}
\begin{enumerate}
\item[(a)]
If $\psi$ is unbounded, and~\eqref{eq:fpt:a1} and~\eqref{eq:fpt:a2} hold for some $t_0 > 0$, then for all $n \ge 0$ and $x \ge 0$,
\formula[eq:fptdensityasymp]{
 \lim_{t \to \infty} \expr{t^{n + 1/2} \, \frac{d^n}{dt^n} \, \pr(\tau_x > t)} & = \frac{(-1)^n \Gamma(n + 1/2)}{\pi} \, V(x) .
}
The convergence is locally uniform in $x \in [0, \infty)$.
\item[(b)]
If $\psi$ is unbounded, regularly varying of order $\ro \in [0, 1]$ at infinity, and~\eqref{eq:fpt:a1} and~\eqref{eq:fpt:a2} hold for some $t_0 > 0$, then for all $n \ge 0$ and $t > t_0$ ($t \ge t_0$ if $n = 0$),
\formula[eq:fptdensityregular]{
 \lim_{x \to 0^+} \expr{\sqrt{\psi(1/x^2)} \, \frac{d^n}{dt^n} \, \pr(\tau_x > t)} & = \frac{(-1)^n \Gamma(n + 1/2)}{\pi \Gamma(1 + \ro)} \, \frac{1}{t^{n + 1/2}} \, .
}
The convergence is uniform in $t \in (t_0, \infty)$.
\end{enumerate}
\end{theorem}

\begin{remark}
The asymptotic behaviour of $(d / dt) \pr(\tau_x > t)$ as $t \to \infty$ was completely described for (possibly asymmetric) stable L\'evy processes in~\cite{bib:ds10}, Theorem~1. In our Theorem~\ref{th:fptdensityasymp}(a), we consider a class of symmetric, but not necessarily stable L\'evy processes, and derivatives of higher order are included.

In~\cite{bib:dr11}, Theorem~3, an analogous problem is studied when one-dimensional distributions of the L\'evy process $X_t$ belong to the domain of attraction of a (possibly asymmetric) stable law. Again, this partially overlaps with Theorem~\ref{th:fptdensityasymp}(a), but neither result generalizes the other one.\qed
\end{remark}

Finally, Theorem~\ref{th:eigenfunctions} is complemented by the following completeness result. It was proved in~\cite{bib:mk10} under an extra assumption that the operator $\Pi$, defined in the statement of the theorem, is injective. Using methods developed partially in~\cite{bib:kmr11}, we show that $\Pi$ is always injective, and therefore the theorem holds in full generality. The present statement was announced in~\cite{bib:mk10} as Theorem~1.3.

\begin{theorem}
\label{th:spectral}
For $f \in C_c(\hl)$ and $\lambda > 0$, let
\formula[eq:pistar]{
 \Pi f(\lambda) & = \int_0^\infty f(x) F_\lambda(x) dx .
}
Then $\sqrt{2 / \pi} \, \Pi$ extends to a unitary operator on $L^2(\hl)$, which diagonalises the action of $P^\hl_t$:
\formula{
 \Pi P^\hl_t f(\lambda) & = e^{-t \psi(\lambda^2)} \Pi f(\lambda) && \text{for $f \in L^2(\hl)$.}
}
Furthermore, $f \in \domain(\A_\hl; L^2)$ if and only if $\psi(\lambda^2) \Pi f(\lambda)$ is in $L^2(\hl)$, and
\formula{
 \Pi \A_\hl f(\lambda) = -\psi(\lambda^2) \Pi f(\lambda) && \text{for $f \in \domain(\A_\hl; L^2)$.}
}
\end{theorem}

\begin{remark}
In~\cite{bib:mk10}, the generalised eigenfunction expansion of Theorem~\ref{th:spectral} was the key step in the proof of (a restricted version of) Theorem~\ref{th:fpt}. Here, the proofs of Theorems~\ref{th:fpt} and~\ref{th:spectral} are independent.\qed
\end{remark}

We conclude the introduction with a brief description of the structure of the article. In Preliminaries, we recall the notion of complete Bernstein and Stieltjes functions, their properties and a Wiener-Hopf type transformation $\psi \mapsto \psi^\dagger$. Some new ideas are developed here, e.g. a type of continuity of the mapping $\psi \mapsto \psi^\dagger$ is proved. We also give some simple estimates related to Laplace transforms of monotone functions, and introduce the renewal function $V(x)$ of the ascending ladder-height process. In Section~\ref{sec:pi2}, we prove Theorem~\ref{th:spectral}. Next two sections contain estimates and properties of $\thet_\lambda$ and $F_\lambda(x)$, respectively (see Theorem~\ref{th:eigenfunctions}), which are essential to the derivation of the main results. In Section~\ref{sec:supcbf}, we prove Theorems~\ref{th:fpt}--\ref{th:fptdensityasymp}. Finally, some examples are studied in Section~\ref{sec:examples}. 

%
%

\section{Preliminaries}
\label{sec:pre}

\subsection{Complete Bernstein functions}
\label{sec:cbf}

A function $f(z)$ is said to be a \emph{complete Bernstein function} (CBF in short) if
\formula[eq:cbf]{
 f(z) & = c_1 + c_2 z + \frac{1}{\pi} \int_{0+}^\infty \frac{z}{z + s} \, \frac{m(ds)}{s} \, ,
}
where $c_1, c_2 \ge 0$, and $m$ is a Radon measure on $(0, \infty)$ such that $\int \min(s^{-1}, s^{-2}) m(ds) < \infty$. A function $g(z)$ is said to be a \emph{Stieltjes function} if
\formula[eq:stieltjes]{
 g(z) & = \frac{c_1}{z} + c_2 + \frac{1}{\pi} \int_{0+}^\infty \frac{1}{z + s} \, \tilde{m}(ds) ,
}
where $c_1, c_2 \ge 0$, and $\tilde{m}$ is a Radon measure on $(0, \infty)$ such that $\int \min(1, s^{-1}) \tilde{m}(ds) < \infty$.

Complete Bernstein and Stieltjes functions are often defined on $(0, \infty)$. However, \eqref{eq:cbf} and~\eqref{eq:stieltjes} define holomorphic functions on $\C \setminus (-\infty, 0]$. We always identify functions on $(0, \infty)$ with their holomorphic extensions to $\C \setminus (-\infty, 0]$.

We list some basic properties of complete Bernstein and Stieltjes functions.

\begin{proposition}[Proposition~2.18 in~\cite{bib:mk10}, and Corollary~6.3 in~\cite{bib:ssv10}]
\label{prop:cbf:repr}
\mbox{}\par
\begin{enumerate}
\item
Let $f$ be a complete Bernstein function with representation~\eqref{eq:cbf}. Then
\formula[eq:cbf:constants]{
 c_1 & = \lim_{z \to 0^+} f(z) , & c_2 & = \lim_{z \to \infty} \frac{f(z)}{z} \, ,
}
and
\formula[eq:jump]{
 m(ds) & = \lim_{\eps \to 0^+} (\im f(-s + i \eps) ds) ,
}
with the limit understood in the sense of weak convergence of measures.
\item
Let $g$ be a Stieltjes function with representation~\eqref{eq:stieltjes}. Then
\formula[eq:stieltjes:constants]{
 c_1 & = \lim_{z \to 0^+} (z g(z)) , & c_2 & = \lim_{z \to \infty} g(z) ,
}
and
\formula[eq:jumpstieltjes]{
 \tilde{m}(ds) + \pi c_1 \delta_0(ds) & = \lim_{\eps \to 0^+} (-\im g(-s + i \eps) ds) ,
}
with the limit understood in the sense of weak convergence of measures.\qed
\end{enumerate}
\end{proposition}

\begin{proposition}[see~\cite{bib:ssv10}]
\label{prop:cbf:prop}
Suppose that $f$ and $g$ are not constantly equal to $0$.
\begin{enumerate}
\item[(a)] the following conditions are equivalent: $f(z)$ is CBF, $z / f(z)$ is CBF, $f(z) / z$ is Stieltjes, $1 / f(z)$ is Stieltjes;
\item[(b)] $f(z)$ is CBF if and only if $f(z) \ge 0$ for $z \ge 0$, and $f$ extends to a holomorphic function in $\C \setminus (-\infty, 0]$ such that $\im f(z) > 0$ when $\im z > 0$;
\item[(c)] if $f$, $g$ are CBF and $a, c > 0$, then also $f(z) + g(z)$, $c f(z)$, $f(c z)$, $(z - a) / (f(z) - f(a))$ (extended continuously at $z = a$) and $f(g(z))$ are CBF.
\qed
\end{enumerate}
\end{proposition}

\begin{proposition}[Proposition~2.21 in~\cite{bib:mk10}]
\label{prop:cbf:ests}
If $f$ is a complete Bernstein function, then
\begin{enumerate}
\item[(a)] $0 \le z f'(z) \le f(z)$ for $z > 0$;
\item[(b)] $0 \le -z f''(z) \le 2 f'(z)$ for $z > 0$;
\item[(c)] $|f(z)| \le (\sin(\eps/2))^{-1} \, f(|z|)$ for $z \in \C$, $|\Arg z| \le \pi - \eps$, $\eps \in (0, \pi)$;
\item[(d)] $|z f'(z)| \le (\sin(\eps/2))^{-1} \, f(|z|)$ for $z$ as in~(c);
\item[(e)] $|f(z)| \le c(f, \eps) (1 + |z|)$ for $z$ as in~(c).\qed
\end{enumerate}
\end{proposition}

Throughout the article, $\psi$ usually denotes the complete Bernstein function in the L\'evy-Khintchine exponent of $X_t$. Some preliminary results and definitions, however, are valid for more general functions $\psi$. If this is the case, we explicitly state all assumptions on $\psi$ each time it is mentioned.

Following~\cite{bib:mk10, bib:kmr11}, we define
\formula[eq:dagger]{
 \psi^\dagger(\xi) & = \exp\expr{\frac{1}{\pi} \int_0^\infty \frac{\xi \log \psi(\zeta^2)}{\xi^2 + \zeta^2} \, d\zeta}
}
for any positive function $\psi$ for which the integral converges. In this case, $\psi^\dagger(\xi)$ is defined at least when $\re \xi > 0$. By a simple substitution, for $\xi > 0$,
\formula[eq:dagger1]{
 \psi^\dagger(\xi) & = \exp\expr{\frac{1}{\pi} \int_0^\infty \frac{\log \psi(\xi^2 \zeta^2)}{1 + \zeta^2} \, d\zeta} .
}

\begin{proposition}[Proposition~2.1 in~\cite{bib:kmr11}]
\label{prop:daggerest}
If $\psi(\xi)$ is a nonnegative function on $(0, \infty)$ and both $\psi(\xi)$ and $\xi / \psi(\xi)$ are increasing on $(0, \infty)$, then
\formula[eq:daggerest]{
 e^{-2 \catalan / \pi} \sqrt{\psi(\xi^2)} & \le \psi^\dagger(\xi) \le e^{2 \catalan / \pi} \sqrt{\psi(\xi^2)} ,
}
where $\catalan \approx 0.916$ is the Catalan constant. Note that $e^{2 \catalan / \pi} \le 2$.

If, in addition, $\psi(\xi)$ is regularly varying at $\infty$, then 
\formula[eq:regvar]{
 \lim_{\xi \to \infty} \frac{\psi^\dagger(\xi)}{\sqrt{\psi(\xi^2)}} & = 1 .
}
An analogous statement for $\xi \to 0$ holds for $\psi(\xi)$ regularly varying at $0$.

In particular, \eqref{eq:daggerest} holds for any CBF. Likewise,~\eqref{eq:regvar} holds for any regularly varying CBF.
\qed
\end{proposition}

The estimate~\eqref{eq:daggerest} for CBFs was obtained independently in~\cite{bib:ksv11}, Proposition~3.7, while~\eqref{eq:regvar} for CBFs was derived in~\cite{bib:ksv9}, Proposition 2.2.

\begin{proposition}[Proposition~3.4 in~\cite{bib:mk10}]
\label{prop:dagger:prop}
Whenever both sides of the following identities make sense, we have:
\begin{enumerate}
\item[(a)] $(1 / \psi)^\dagger = 1 / \psi^\dagger$, $(\psi^\ro)^\dagger = (\psi^\dagger)^\ro$ ($\ro \in \R$) and $(\psi_1 \psi_2)^\dagger = \psi_1^\dagger \psi_2^\dagger$;
\item[(b)] $(c^2 \psi)^\dagger = c \psi^\dagger$ for $c > 0$;
\item[(c)] when $\psi(\xi) = \xi$, then $\psi^\dagger(\xi) = \xi$;
\item[(d)] if appropriate limits of $\psi$ exist, then the corresponding limits of $\psi^\dagger$ exist, and $\lim_{\xi \to 0^+} \psi^\dagger(\xi) = (\lim_{\xi \to 0^+} \psi(\xi))^{1/2}$, $\lim_{\xi \to \infty} \psi^\dagger(\xi) = (\lim_{\xi \to \infty} \psi(\xi))^{1/2}$. \qed
\end{enumerate}
\end{proposition}

The first part of the following result was independently proved in~\cite{bib:ksv10}, Proposition~2.4.

\begin{lemma}[Lemma~3.8 in~\cite{bib:mk10}]
\label{lem:duality}
If $\psi(\xi)$ is a CBF, then also $\psi^\dagger(\xi)$ is a CBF, and
\formula[eq:duality]{
 \psi^\dagger(\xi) \psi^\dagger(-\xi) & = \psi(-\xi^2)
}
for $\xi \in \C \setminus \R$.
\qed
\end{lemma}

\begin{proposition}
\label{prop:daggercontinuity}
Let $\psi_n$ be a sequence of CBFs. If $\psi_n(\xi) \to \psi(\xi)$ as $n \to \infty$ for all $\xi > 0$, then $\psi_n^\dagger(\xi) \to \psi^\dagger(\xi)$ locally uniformly in $\xi \in \C \setminus (-\infty, 0]$.
\end{proposition}

\begin{proof}
For any CBFs $\psi(\xi)$, $\tilde{\psi}(\xi)$, we have
\formula[eq:daggercontinuity]{
 \frac{\tilde{\psi}^\dagger(\xi)}{\psi^\dagger(\xi)} & = \exp\expr{\frac{1}{\pi} \int_0^\infty \frac{\xi \log (\tilde{\psi}(\zeta^2) / \psi(\zeta^2))}{\xi^2 + \zeta^2} \, d\zeta}
}
when $\re \xi > 0$. Furthermore, by monotonicity and concavity of $\psi_n$, $\psi_n(1) \min(1, \zeta^2) \le \psi_n(\zeta^2) \le \psi_n(1) \max(1, \zeta^2)$ ($\zeta > 0$), and therefore there are $c_1, c_2 > 0$ such that $c_1 \min(1, \zeta^2) \le \psi_n(\zeta^2) \le c_2 \max(1, \zeta^2)$. Hence, by dominated convergence,~\eqref{eq:daggercontinuity} gives convergence of $\psi_n^\dagger(\xi)$ to $\psi^\dagger(\xi)$ when $\re \xi > 0$. By Corollary~7.6(b) in~\cite{bib:ssv10}, $\psi_n^\dagger(\xi) \to \psi^\dagger(\xi)$ for all $\xi \in \C \setminus (-\infty, 0]$. Pointwise convergence of CBFs is automatically locally uniform on $\C \setminus (-\infty, 0]$: by Proposition~\ref{prop:cbf:ests}(d,e), when $\Arg \xi \in (-\pi + \eps, \pi - \eps)$, we have $|(\psi_n^\dagger)'(\xi)| \le c_3(\eps) \psi_n^\dagger(|\xi|) / |\xi| \le c_4(\eps, f) \max(|\xi|^{-1}, 1)$, which proves that $\psi_n^\dagger$ are locally equicontinuous in $\C \setminus (-\infty, 0]$.
\end{proof}

As in~\cite{bib:mk10}, for a nonvanishing, increasing and differentiable function $\psi$ with strictly positive derivative, we denote
\formula[eq:psilambda]{
 \psi_\lambda(\xi) & = \frac{1 - \xi / \lambda^2}{1 - \psi(\xi) / \psi(\lambda^2)} \, , && \lambda, \xi > 0 , \, \lambda^2 \ne \xi .
}
This definition is extended continuously by $\psi_\lambda(\lambda^2) = \psi(\lambda^2) / (\lambda^2 \psi'(\lambda^2))$, and when $\psi$ extends holomorphically to $\C \setminus (-\infty, 0]$, also $\psi_\lambda(\xi)$ is defined for $\xi \in \C \setminus (-\infty, 0]$. For simplicity, we denote $\psi_\lambda^\dagger(\xi) = (\psi_\lambda)^\dagger(\xi)$. By Proposition~\ref{prop:cbf:prop}(c), if $\psi(\xi)$ is a CBF, then $\psi_\lambda(\xi)$ (and hence also $\psi_\lambda^\dagger$) is a CBF for any $\lambda > 0$.

Although~\eqref{eq:psilambda} is generally used for complete Bernstein functions $\psi$, we will occasionaly need this definition for a \emph{negative} function $\psi(\xi) = -\xi^\ro$ with $\ro \in (-1, 0)$.

\begin{corollary}
\label{cor:psilambdadaggerest}
If $\psi(\xi)$ is a CBF, then
\formula{
 e^{-2 \catalan / \pi} \sqrt{\psi_\lambda(\xi^2)} & \le \psi_\lambda^\dagger(\xi) \le e^{2 \catalan / \pi} \sqrt{\psi_\lambda(\xi^2)} . \qed
}
\end{corollary}

\begin{corollary}
\label{cor:psilambdacontinuity}
If $\psi$ is a CBF, then the function $\psi_\lambda^\dagger(\xi)$ is jointly continuous in $\lambda > 0$, $\xi \in \C \setminus (-\infty, 0)$.\qed
\end{corollary}

The following monotonicity property of $\psi_\lambda(\xi)$ plays an important role.

\begin{proposition}
\label{prop:psilambdaest}
Suppose that for twice differentiable, nonvanishing and increasing functions $\psi(\xi)$ and $\tilde{\psi}(\xi)$ we have
\formula[eq:secondinequality]{
 \frac{-\psi''(\zeta)}{\psi'(\zeta)} & \le \frac{-\tilde{\psi}''(\zeta)}{\tilde{\psi}'(\zeta)} \/ , & \zeta > 0 .
}
Then
\formula{
 \frac{\psi_\lambda(\xi^2)}{\psi_\lambda(\lambda^2)} \le \frac{\tilde{\psi}_\lambda(\xi^2)}{\tilde{\psi}_\lambda(\lambda^2)} \, , && 0 < \lambda < \xi , \\
 \frac{\psi_\lambda(\xi^2)}{\psi_\lambda(\lambda^2)} \ge \frac{\tilde{\psi}_\lambda(\xi^2)}{\tilde{\psi}_\lambda(\lambda^2)} \, , && 0 < \xi < \lambda .
}
\end{proposition}

\begin{proof}
Integration in $\zeta$ of both sides of~\eqref{eq:secondinequality} yields that
\formula{
 \frac{\psi'(\zeta)}{\psi'(\xi_1)} & \le \frac{\tilde{\psi}'(\zeta)}{\tilde{\psi}'(\xi_1)} \/ , && 0 < \xi_1 < \zeta ,
}
and by another integration in $\zeta$,
\formula{
 \frac{\psi(\xi_2) - \psi(\xi_1)}{\psi'(\xi_1)} & \le \frac{\tilde{\psi}(\xi_2) - \tilde{\psi}(\xi_1)}{\tilde{\psi}'(\xi_1)} \/ , && 0 < \xi_1 < \xi_2 .
}
Taking $\xi_1 = \lambda^2$ and $\xi_2 = \xi^2$, we obtain
\formula{
 \frac{\psi_\lambda(\xi^2)}{\psi_\lambda(\lambda^2)} & = \frac{\psi'(\lambda^2) (\lambda^2 - \xi^2)}{\psi(\lambda^2) - \psi(\xi^2)} \ge \frac{\tilde{\psi}'(\lambda^2) (\lambda^2 - \xi^2)}{\tilde{\psi}(\lambda^2) - \tilde{\psi}(\xi^2)} = \frac{\tilde{\psi}_\lambda(\xi^2)}{\tilde{\psi}_\lambda(\lambda^2)} \, , && 0 < \xi < \lambda .
}
The other part is proved in a similar manner.
\end{proof}

Finally, recall that $\psi_\lambda(\lambda^2) = \psi(\lambda^2) / (\lambda^2 \psi'(\lambda^2))$. This gives the following simple estimate:
\formula[eq:psilambdaest3]{
 \frac{\psi_\lambda(\xi^2)}{\psi_\lambda(\lambda^2)} & = \frac{\psi'(\lambda^2) (\lambda^2 - \xi^2)}{\psi(\lambda^2) - \psi(\xi^2)} \le \frac{\psi'(\lambda^2) \max(\lambda^2, \xi^2)}{|\psi(\lambda^2) - \psi(\xi^2)|} \/ , & \lambda, \xi > 0
}
when $\psi$ is a CBF.

Let $\psi(\xi)$ be the complete Bernstein function related with the L\'evy-Khintchine exponent $\Psi(\xi)$ of the L\'evy process $X_t$, $\Psi(\xi) = \psi(\xi^2)$. With the notation of this section, for $\lambda, \xi > 0$ we have (see Remark~4.12 in~\cite{bib:mk10}, and the proof of Theorem~4.1 in~\cite{bib:kmr11})
\formula[eq:ltau]{
 \int_0^\infty e^{-\xi x} \pr(\tau_x > t) dx = \frac{2}{\pi} \int_0^\infty \frac{\lambda}{\lambda^2 + \xi^2} \, \frac{\lambda \psi'(\lambda^2)}{\psi(\lambda^2)} \, \psi_\lambda^\dagger(\xi) e^{-t \psi(\lambda^2)} d\lambda ;
}
\formula[eq:theta0]{
 \thet_\lambda & = \frac{1}{\pi} \int_0^\infty \frac{\lambda}{\lambda^2 - \zeta^2} \, \log \frac{\psi_\lambda(\zeta^2)}{\psi_\lambda(\lambda^2)} \, d\zeta = \Arg \psi_\lambda^\dagger(i \lambda) ;
}
\formula[eq:lf0]{
 \laplace F_\lambda(\xi) &= \frac{\lambda}{\lambda^2 + \xi^2} \, \frac{\psi_\lambda^\dagger(\xi)}{\sqrt{\psi_\lambda(\lambda^2)}} \, ;
}
\formula[eq:gamma0]{
 \gamma_\lambda(d\xi) &= \frac{1}{\pi} \frac{\sqrt{\psi_\lambda(\lambda^2)}}{\psi_\lambda^\dagger(\xi)} \, \im \frac{\lambda \psi'(\lambda^2)}{\psi(\lambda^2) - \psi^+(-\xi^2)} \, d\xi ;
}
for the last equality the assumption of the final part of Theorem~\ref{th:eigenfunctions} is required.

%
%

\subsection{Estimates for the Laplace transform}
\label{sec:lap}

This short section contains some rather standard estimates for the inverse Laplace transform, similar to those used in~\cite{bib:kmr11}.

\begin{proposition}
\label{prop:lb}
Let $a > 0$, $c \ge 1$. If $f$ is nonnegative and $f(x) \le c f(a) \max(1, x/a)$ \textrm{($x > 0$)}, then for any $\xi > 0$,
\formula{
 f(a) & \ge \frac{\xi \laplace f(\xi)}{c (1 + (a \xi)^{-1} e^{-a \xi})} \/ .
}
\end{proposition}

\begin{proof}
We have
\formula{
 \xi \laplace f(\xi) & = \int_0^a \xi e^{-\xi x} f(x) dx + \int_a^\infty \xi e^{-\xi x} f(x) dx \\
 & \le c f(a) \int_0^a \xi e^{-\xi x} dx + \frac{c f(a)}{a} \int_a^\infty \xi x e^{-\xi x} dx \\
 & = c f(a) (1 - e^{-a \xi}) + \frac{c f(a)}{a \xi} \, (1 + a \xi) e^{-a \xi} \\
 & = c f(a) (1 + (a \xi)^{-1} e^{-a \xi}) . \qedhere
}
\end{proof}

\begin{proposition}
\label{prop:ub}
If $f$ is nonnegative and increasing, then for $a, \xi > 0$,
\formula{
 f(a) & \le e^{a \xi} \xi \laplace f(\xi) \/ .
}
\end{proposition}

\begin{proof}
As before,
\formula{
 \xi \laplace f(\xi) & = \int_0^a \xi e^{-\xi x} f(x) dx + \int_a^\infty \xi e^{-\xi x} f(x) dx \\
 & \ge f(a) \int_a^\infty \xi e^{-\xi x} dx = f(a) e^{-a \xi} . \qedhere
}
\end{proof}

\begin{proposition}
\label{prop:ub2}
Let $b \ge a > 0$, and suppose that $f$ is nonnegative and increasing on $(0, b)$, and $f(x) \ge m$ for $x \ge b$. Then for any $\xi > 0$,
\formula{
 f(a) & \le \frac{e^{a \xi} \xi \laplace f(\xi) - m e^{-(b - a) \xi}}{1 - e^{-(b - a) \xi}} .
}
\end{proposition}

\begin{proof}
By Proposition~\ref{prop:ub} applied to $f(x) \ind_{(0, a)}(x) + f(a) \ind_{[a, \infty)}(x)$,
\formula{
 f(a) & \le e^{a \xi} \xi \expr{\int_0^\infty e^{-\xi x} f(x) dx + \int_a^\infty e^{-\xi x} (f(a) - f(x)) dx} \\
 & \le e^{a \xi} \xi \expr{\int_0^\infty e^{-\xi x} f(x) dx + \int_b^\infty e^{-\xi x} (f(a) - m) dx} \\
 & = e^{a \xi} \xi \laplace f(\xi) + e^{-(b - a) \xi} (f(a) - m) . \qedhere
}
\end{proof}

%
%

\subsection{Elements of fluctuation theory}
\label{sec:v}

In the sequel we will need the function $V(x)$, which is described by its Laplace transform, $\laplace V(\xi) = \xi / \psi^\dagger(\xi)$. By Proposition~\ref{prop:cbf:prop}(a) and Lemma~\ref{lem:duality}, $\xi / \psi^\dagger(\xi)$ is a complete Bernstein function. Hence, $V$ is a Bernstein function, i.e. $V$ is nonnegative and $V'$ is completely monotone (see~\cite{bib:ssv10}). By Theorem~4.4 in~\cite{bib:kmr11},
\formula{
 \frac{1}{5} \, \frac{1}{\sqrt{\psi(1/x^2)}} \le V(x) & \le 5 \, \frac{1}{\sqrt{\psi(1/x^2)}} \, , && x > 0 .
}
Suppose that $\psi$ extends to a function $\psi^+$ holomorphic in the upper complex half-plane and continuous in the region $\im \xi \ge 0$. Then by Proposition~4.5 in~\cite{bib:kmr11},
\formula{
 V(x) & = b x + \frac{1}{\pi} \int_{0^+}^\infty \im \expr{-\frac{1}{\psi^+(-\xi^2)}} \frac{\psi^\dagger(\xi)}{\xi} \, (1 - e^{-x \xi}) d\xi , && x > 0 ,
}
where $b = (\lim_{\xi \to 0^+} (\xi / \psi(\xi)))^{1/2}$.

Probabilistically, $V(x)$ is the renewal function for the ascending ladder-height process $H_s$ corresponding to $X_t$. When $\psi$ is unbounded, then $X_t$ satisfies the absolute continuity condition and $V(x)$ is the (unique up to a multiplicative constant) increasing harmonic function for $X_t$ on $(0, \infty)$, cf.~\cite{bib:s80}.

%
%

\section{Eigenfunction expansion}
\label{sec:pi2}

In this section we prove a peculiar integral identity, which (together with the results of~\cite{bib:mk10}) yields a short proof of Theorem~\ref{th:spectral}.

\begin{lemma}
\label{lem:piovertwo}
Let $\psi(\xi)$ be an arbitrary positive, continuously differentiable function on $(0, \infty)$ with strictly positive derivative, and such that $(\log (1 + \psi(\zeta^2))) / \zeta^2$ is integrable on $(0, \infty)$. Then for all $\xi_1, \xi_2 > 0$,
\formula[eq:piovertwo]{
 \int_0^\infty \frac{\lambda^4 \psi'(\lambda^2) (\xi_1 + \xi_2) \psi_\lambda^\dagger(\xi_1) \psi_\lambda^\dagger(\xi_2)}{\psi(\lambda^2) (\lambda^2 + \xi_1^2) (\lambda^2 + \xi_2^2)} \, d\lambda & = \frac{\pi}{2} \/ .
}
\end{lemma}

\begin{proof}
As in the proof of Theorem~4.1 in~\cite{bib:kmr11}, for $\xi > 0$ and $z \in \C \setminus (-\infty, 0]$ we define
\formula{
 \ph(\xi, z) & = \exp\expr{-\frac{1}{\pi} \int_0^\infty \frac{\xi \log (1 + \psi(\zeta^2) / z)}{\xi^2 + \zeta^2} \, d\zeta} ,
}
where $\log$ is the principal branch of the logarithm. Clearly, for each $\xi > 0$, $\ph(\xi, z)$ is a holomorphic function of $z$, and $\ph(\xi, z) > 0$ when $z > 0$. As it was observed in~\cite{bib:kmr11}, when $\im z > 0$, we have $\Arg (1 + \psi(\zeta^2) / z) \in (-\pi, 0)$, so that
\formula{
 \Arg \ph(\xi, z) & = -\frac{1}{\pi} \int_0^\infty \frac{\xi \Arg (1 + \psi(\zeta^2) / z)}{\xi^2 + \zeta^2} \, d\zeta \in (0, \pi/2) .
}
Hence, for any $\xi_1, \xi_2 > 0$, $\im (\ph(\xi_1, z) \ph(\xi_2, z)) > 0$ when $\im z > 0$. By Proposition~\ref{prop:cbf:prop}(b), $f(z) = \ph(\xi_1, z) \ph(\xi_2, z)$ is a CBF of $z$.

By monotone convergence, $f(z)$ converges to $0$ when $z \to 0^+$, and by dominated convergence, $\lim_{z \to \infty} f(z) = 1$. It follows that the constants $c_1$, $c_2$ in the representation~\eqref{eq:cbf} for the CBF $f(z)$ vanish. Furthermore, for any $\xi > 0$, $\ph(\xi, z)$ extends to a continuous function $\ph^+(\xi, z)$ in the region $\im z \ge 0$, given by the formula
\formula{
 \ph^+(\xi, z) & = \exp\expr{-\frac{1}{\pi} \int_0^\infty \frac{\xi \log^- (1 + \psi(\zeta^2) / z)}{\xi^2 + \zeta^2} \, d\zeta} ,
}
where $\log^-(z)$ is the continuous extension of $\log z$ to the region $\im z \le 0$. The measure $m$ in the representation~\eqref{eq:cbf} for the function $f(z) = \ph(\xi_1, z) \ph(\xi_2, z)$ is therefore given by $m(ds) = \pi^{-1} \im(\ph^+(\xi_1, -s) \ph^+(\xi_2, -s)) ds$ (Proposition~\ref{prop:cbf:repr}(a)). By~\eqref{eq:cbf} and monotone convergence,
\formula[eq:piovertwo0]{
\begin{aligned}
 \hspace*{2em} & \hspace*{-2em} \frac{1}{\pi} \int_0^\infty \frac{\im (\ph^+(\xi_1, -s) \ph^+(\xi_2, -s))}{s} \, ds \\
 & = \frac{1}{\pi} \, \lim_{z \to \infty} \int_0^\infty \frac{z \im (\ph^+(\xi_1, -s) \ph^+(\xi_2, -s))}{z + s} \, \frac{ds}{s} \\
 & = \lim_{z \to \infty} (\ph(\xi_1, z) \ph(\xi_2, z)) = 1 .
\end{aligned}
}
Note that if $\psi(\xi)$ is bounded on $(0, \infty)$ and $s \ge \sup_{\xi > 0} \psi(\xi)$, then $\ph^+(\xi, -s)$ is real, so that the integrand in~\eqref{eq:piovertwo0} vanishes for such $s$. Hence, we may substitute $\psi(\lambda^2)$ for $s$ in~\eqref{eq:piovertwo0} to obtain that
\formula{
 \frac{2}{\pi} \int_0^\infty \lambda \psi'(\lambda^2) \im \frac{\ph^+(\xi_1, -\psi(\lambda^2)) \ph^+(\xi_2, -\psi(\lambda^2))}{\psi(\lambda^2)} \, d\lambda & = 1 .
}
Finally, in the proof of Theorem~4.1 in~\cite{bib:kmr11} (formula~(4.3) therein) it is proved that
\formula{
 \ph^+(\xi, -\psi(\lambda^2)) & = \frac{\lambda (\lambda + \xi i) \psi_\lambda^\dagger(\xi)}{\lambda^2 + \xi^2} \, .
}
Since $\im ((\lambda + \xi_1 i) (\lambda + \xi_2 i)) = \lambda (\xi_1 + \xi_2)$, the lemma is proved.
\end{proof}

\begin{remark}
Lemma~\ref{lem:piovertwo} was conjectured in a preliminary version of~\cite{bib:mk10}.\qed
\end{remark}

\begin{proof}[Proof of Theorem~\ref{th:spectral}]
In~\cite{bib:mk10}, Theorem~\ref{th:spectral} was proved under an extra assumption that the operator $\Pi$ defined in~\eqref{eq:pistar} is injective. Below we show that this condition is always satisfied.

Let $e_\xi(x) = e^{-\xi x} \ind_\hl(x)$ ($\xi > 0$). By~\eqref{eq:lf0}, for $\lambda > 0$,
\formula{
 \Pi e_\xi(\lambda) = \int_0^\infty e^{-\xi x} F_\lambda(x) dx = \laplace F_\lambda(\xi) = \sqrt{\frac{\lambda^2 \psi'(\lambda^2)}{\psi(\lambda^2)}} \, \frac{\lambda}{\lambda^2 + \xi^2} \, \psi_\lambda^\dagger(\xi) .
}
By Lemma~\ref{lem:piovertwo},
\formula{
 \scalar{\Pi e_{\xi_1}, \Pi e_{\xi_2}}_{L^2(\hl)} & = \int_0^\infty \frac{\lambda^4 \psi'(\lambda^2) \psi_\lambda^\dagger(\xi_1) \psi_\lambda^\dagger(\xi_2)}{\psi(\lambda^2) (\lambda^2 + \xi_1^2) (\lambda^2 + \xi_2^2)} \, d\lambda \\
 & = \frac{\pi}{2 (\xi_1 + \xi_2)} = \frac{\pi}{2} \scalar{e_{\xi_1}, e_{\xi_2}}_{L^2(\hl)} .
}
The functions $e_\xi$ ($\xi > 0$) form a linearly dense set in $L^2((0, \infty))$. By approximation, $\sqrt{2 / \pi} \, \Pi$ is a unitary operator on $L^2(\hl)$. In particular, $\Pi$ is injective.
\end{proof}

%
%

\section{Estimates of $\thet_\lambda$}
\label{sec:theta}

As it will become clear in the next section, upper bounds for the phase shift $\thet_\lambda$ (which is defined by~\eqref{eq:theta}) are crucial for applications of Theorems~\ref{th:eigenfunctions} and~\ref{th:spectral}: when $\thet_\lambda$ is close to $\pi/2$, estimates of the eigenfunctions $F_\lambda(x)$ are problematic. Recall that by~\eqref{eq:theta} and~\eqref{eq:theta0},
\formula[eq:theta1]{
 \thet_\lambda & = \Arg \psi_\lambda^\dagger(i \lambda) = -\frac{1}{\pi} \int_0^\infty \frac{1}{1 - \zeta^2} \, \log \frac{\psi_\lambda(\lambda^2 \zeta^2)}{\psi_\lambda(\lambda^2)} \, d\zeta , && \lambda > 0 .
}
Here $\psi(\xi)$ is a complete Bernstein function in the L\'evy-Khintchine exponent of $X_t$. By a substitution $\zeta = 1 / s$ in the integral over $(1, \infty)$, one obtains that (see Proposition~4.16 in~\cite{bib:mk10})
\formula[eq:theta2]{
 \thet_\lambda & = \frac{1}{\pi} \int_0^1 \frac{1}{1 - \zeta^2} \, \log \frac{\psi_\lambda(\lambda^2 / \zeta^2)}{\psi_\lambda(\lambda^2 \zeta^2)} \, d\zeta , && \lambda > 0 .
}
When $\psi(\xi) = \xi^{\alpha/2}$ ($0 < \alpha \le 2)$, that is, when $X_t$ is the symmetric $\alpha$-stable process, we have $\thet_\lambda = (2 - \alpha) \pi/8$ (Example~6.1 in \cite{bib:mk10}). In the general case, $0 \le \thet_\lambda < \pi/2$, and $\thet_\lambda \le (\pi/2) \sup_{\xi > 0} (\xi \psi_\lambda'(\xi) / \psi_\lambda(\xi))$ (Proposition~4.18 in \cite{bib:mk10}). This estimate is insufficient for our needs, and it is significantly improved below. By~\eqref{eq:theta1} and Proposition~\ref{prop:psilambdaest}, we have the following monotonicity result. Note that $\psi$ need not be a complete Bernstein function in Propositions~\ref{prop:thetamonotone} and~\ref{prop:int}.

\begin{proposition}
\label{prop:thetamonotone}
Suppose that $\psi$ and $\tilde{\psi}$ are arbitrary twice-differentiable, increasing and nonvanishing functions (not necessarily complete Bernstein functions). Define $\thet_\lambda$ by~\eqref{eq:theta2}, and define $\tilde{\thet}_\lambda$ in a similar way, using $\tilde{\psi}$. We assume that the integrals in the definitions of $\thet_\lambda$ and $\tilde{\thet}_\lambda$ are convergent and finite. If $-\psi''(\xi) / \psi'(\xi) \le -\tilde{\psi}''(\xi) / \tilde{\psi}'(\xi)$ for all $\xi > 0$, then $\thet_\lambda \le \tilde{\thet}_\lambda$ for all $\lambda > 0$. \qed
\end{proposition}

Taking $\tilde{\psi}(\xi) = \xi^{\alpha/2}$ (corresponding to the symmetric $\alpha$-stable process) and using the fact that $\tilde{\thet}_\lambda = (2 - \alpha) \pi / 8$, we obtain a global estimate for $\thet_\lambda$. It turns out that although this method works for $\alpha \in (0, 2]$, it generalises easily to $\alpha \in [-2, 2]$. First, we recall the following result, which was given in Example~6.1 in~\cite{bib:mk10} for $\ro \in (0, 1]$; the proof, however, works for any $\ro \ne 0$.

\begin{proposition}[see Example~6.1 in~\cite{bib:mk10}]
\label{prop:int}
Let $\ro \ne 0$ and $\psi(\xi) = \xi^\ro$ if $\ro > 0$, $\psi(\xi) = -\xi^\ro$ if $\ro < 0$. Define $\thet_\lambda$ by~\eqref{eq:theta2}. Then $\thet_\lambda = (1 - \ro) \pi / 4$.\qed
\end{proposition}

From now on, $\psi$ is again the complete Bernstein function such that $\Psi(\xi) = \psi(\xi^2)$ is the L\'evy-Khintchine exponent of the process $X_t$.

\begin{proposition}
\label{prop:thetaest2}
We have
\formula[eq:thetaest2]{
 \expr{\inf_{\xi > 0} \frac{\xi |\psi''(\xi)|}{\psi'(\xi)}} \frac{\pi}{4} & \le \thet_\lambda \le \expr{\sup_{\xi > 0} \frac{\xi |\psi''(\xi)|}{\psi'(\xi)}} \frac{\pi}{4} \/ .
}
\end{proposition}

\begin{proof}
Denote the supremum in~\eqref{eq:thetaest2} by $1 - \ro$. By Proposition~\ref{prop:cbf:ests}(b), $\ro \in [-1, 1]$. Suppose that $\ro$ is non-zero, and let $\tilde{\psi}(\xi) = \xi^\ro$ if $\ro > 0$, $\tilde{\psi}(\xi) = -\xi^\ro$ if $\ro < 0$. Observe that $\tilde{\psi}$ is increasing and $-\xi \tilde{\psi}''(\xi) / \tilde{\psi}'(\xi) = (1 - \ro)$, so that
\formula{
 \frac{-\psi''(\zeta)}{\psi'(\zeta)} & \le \frac{-\tilde{\psi}''(\zeta)}{\tilde{\psi}'(\zeta)} \, , && \zeta > 0 .
}
The upper bound in~\eqref{eq:thetaest2} follows by Propositions~\ref{prop:thetamonotone} and~\ref{prop:int}. When $\ro = 0$, we simply consider $\tilde{\psi}(\xi) = -\xi^{-\eps}$ and let $\eps \to 0^+$. Finally, the lower bound is proved in a similar manner.
\end{proof}

We conclude this section with a \emph{local} estimate of $\thet_\lambda$, which depends solely on $\psi(\lambda^2)$, $\psi'(\lambda^2)$ and $\psi''(\lambda^2)$. This result is used only in the construction of an irregular example in Subsection~\ref{subsec:irregular}. We need the following technical result.

\begin{proposition}
\label{prop:integral}
For $a > 0$,
\formula{
 -\frac{1}{\pi} \int_0^\infty \frac{1}{1 - x^2} \, \log \frac{1 + a^2 x^2}{1 + a^2} \, dx & = \arctan a .
}
For $b \in (0, 1)$,
\formula{
 -\frac{1}{\pi} \int_0^1 \frac{1}{1 - x^2} \, \log (1 - b^2 (1 - x^2)) dx & = \frac{(\arcsin b)^2}{\pi} .
}
\end{proposition}

\begin{proof}
The function $\log((1 + a^2 x^2) / (1 + a^2))$ is holomorphic in the upper half-plane with a branch cut along $(i a^{-1}, i \infty)$. By an appropriate contour integration and limit procedure (we omit the details),
\formula{
 & 2 \int_0^\infty \frac{1}{1 - x^2} \, \log \frac{1 + a^2 x^2}{1 + a^2} \, dx = \int_{-\infty}^\infty \frac{1}{1 - x^2} \, \log \frac{1 + a^2 x^2}{1 + a^2} \, dx \\
 & \qquad = \int_{i a^{-1}}^{i \infty} \frac{2 \pi i}{1 - x^2} \, dx = -2 \pi \int_{1/a}^\infty \frac{1}{1 + y^2} \, dy = -2 \pi \arctan a .
}
For the second equality, we have by Taylor expansion and beta integral,
\formula{
 & -\int_0^\infty \frac{1}{1 - x^2} \, \log (1 - b^2 (1 - x^2)) dx = \sum_{n = 1}^\infty \frac{b^{2n}}{n} \int_0^1 (1 - x^2)^{n-1} dx \\
 & \qquad = \sum_{n = 1}^\infty \frac{b^{2n}}{2 n} \int_0^1 (1 - s)^{n-1} s^{-1/2} ds = \sum_{n = 1}^\infty \frac{\Gamma(n) \Gamma(1/2) b^{2n}}{2 n \Gamma(n + 1/2)} = (\arcsin b)^2 ;
}
see e.g.~p.~108 in~\cite{bib:m73} for the last identity.
\end{proof}

\begin{proposition}
\label{prop:thetaest1}
We have
\formula{
 \thet_\lambda & \le \frac{\pi}{2} - \arcsin \sqrt{\frac{\lambda^2 \psi'(\lambda^2)}{\psi(\lambda^2)}}
}
and
\formula{
 \thet_\lambda & \ge \expr{\arcsin \sqrt{\frac{\lambda^2 |\psi''(\lambda^2)|}{2 \psi'(\lambda^2)}}}^2 + \expr{\arcsin \sqrt{\frac{\lambda^2 \psi(\lambda^2) |\psi''(\lambda^2)|}{2 \psi'(\lambda^2) (\psi(\lambda^2) - \lambda^2 \psi'(\lambda^2))}}}^2 \\
 & \hspace*{5em} - \expr{\arcsin \sqrt{\frac{\lambda^4 |\psi''(\lambda^2)|}{2 (\psi(\lambda^2) - \lambda^2 \psi'(\lambda^2))}}}^2 .
}
\end{proposition}

It is easy to see that these bounds are sharp for $\psi(\xi) = c_1 \xi / (c_2 + \xi)$ and $\psi(\xi) = c \xi$ (and, in fact, these are all CBFs $\psi$ for which equalities hold), see examples in Section~\ref{sec:examples}.

\begin{proof}
Note that $\psi_\lambda(0) = 1$ and $\psi_\lambda(\lambda^2) = \psi(\lambda^2) / (\lambda^2 \psi'(\lambda^2))$. Let $a^2 = \psi_\lambda(\lambda^2) - 1$. Since $\psi_\lambda$ is concave, we have $\psi_\lambda(\lambda^2 \zeta^2) \ge 1 + a^2 \zeta^2$ for $\zeta < 1$, and $\psi_\lambda(\lambda^2 \zeta^2) \le 1 + a^2 \zeta^2$ for $\zeta > 1$. Hence, by Proposition~\ref{prop:integral},
\formula{
 \thet_\lambda & = -\frac{1}{\pi} \int_0^\infty \frac{1}{1 - \zeta^2} \log \frac{\psi_\lambda(\lambda^2 \zeta^2)}{\psi_\lambda(\lambda^2)} \, d\zeta \\
 & \le -\frac{1}{\pi} \int_0^\infty \frac{1}{1 - \zeta^2} \log \frac{1 + a^2 \zeta^2}{1 + a^2} \, d\zeta = \arctan a .
}
Since $\arctan a = \pi/2 - \arcsin (1 / (1 + a^2)^{1/2})$, this proves the upper bound. 

For the lower bound, let $a_1^2 = \psi_\lambda(\lambda^2) = 1 + a^2$ and $a_2^2 = \lambda^2 \psi_\lambda'(\lambda^2) / \psi_\lambda(\lambda^2)$. By a short calculation, $a_2^2 = \lambda^2 |\psi''(\lambda^2)| / (2 \psi'(\lambda^2))$. Since $\psi_\lambda(\xi)$ is concave, its graph lies below the tangent line at $\xi = \lambda^2$, so that
\formula{
 \psi_\lambda(\lambda^2 \zeta^2) & \le \psi_\lambda(\lambda^2) - \lambda^2 \psi_\lambda'(\lambda^2) (1 - \zeta^2) \\
 & = a_1^2 (1 - a_2^2 (1 - \zeta^2)) .
}
A similar argument for the CBF $\xi / (\psi_\lambda(\xi) - 1)$ and a short calculation give
\formula{
 \frac{\xi}{\psi_\lambda(\xi) - 1} & \le \frac{\lambda^2}{a^2} + \frac{a^2 - a_1^2 a_2^2}{a^4} \, (\xi - \lambda^2) \\
 & = \frac{a^2 \xi - a_1^2 a_2^2 (\xi - \lambda^2)}{a^4} \, ,
}
which for $\xi = \lambda^2 / \zeta^2$ ($\zeta > 0$) reduces to
\formula{
 \psi_\lambda(\lambda^2 / \zeta^2) & \ge 1 + \frac{a^4 \zeta^{-2}}{a^2 \zeta^{-2} - a_1^2 a_2^2 (\zeta^{-2} - 1)} \\
 & = \frac{a_1^2 (a^2 - a_2^2 (1 - \zeta^2))}{a^2 + a_1^2 a_2^2 (1 - \zeta^2)} \, .
}
Therefore, by~\eqref{eq:theta2},
\formula{
 \thet_\lambda & = \frac{1}{\pi} \int_0^1 \frac{1}{1 - \zeta^2} \log \frac{\psi_\lambda(\lambda^2 / \zeta^2)}{\psi_\lambda(\lambda^2 \zeta^2)} \, d\zeta \\
 & \ge \frac{1}{\pi} \int_0^1 \frac{1}{1 - \zeta^2} \log \frac{1 - a^{-2} a_2^2 (1 - \zeta^2)}{(1 - a_2^2 (1 - \zeta^2)) (1 - a^{-2} a_1^2 a_2^2 (1 - \zeta^2))} \, d\zeta .
}
Triple application of Proposition~\ref{prop:integral} yields
\formula{
 \pi \thet_\lambda & \ge (\arcsin a_2)^2 + (\arcsin(a_1 a_2 / a))^2 - (\arcsin(a_2/a))^2 .
\qedhere
}
\end{proof}

%
%

\section{Properties of $F_\lambda(x)$}
\label{sec:flambda}

Some basic estimates of the eigenfunctions $F_\lambda(x)$ have already been established in~\cite{bib:mk10}. However, for the proof of Theorems~\ref{th:fpt}--\ref{th:fptdensityasymp}, more detailed properties of $F_\lambda(x)$ are required. Recall that by~\eqref{eq:lf} and~\eqref{eq:lf0},
\formula[eq:lf1]{
 \laplace F_\lambda(\xi) & =  \frac{\lambda}{\lambda^2 + \xi^2} \, \frac{\psi_\lambda^\dagger(\xi)}{\sqrt{\psi_\lambda(\lambda^2)}} \/ , && \lambda > 0, \, \re \xi > 0 ,
}
and by~\eqref{eq:f},
\formula{
 F_\lambda(x) & = \sin(\lambda x + \thet_\lambda) - G_\lambda(x) , && \lambda > 0 , \, x > 0 .
}
The phase shift $\thet_\lambda$ was studied in detail in the previous section. Recall that the correction term $G_\lambda(x)$ is a completely monotone function, $G_\lambda(x) = \laplace \gamma_\lambda(x)$, and $\gamma_\lambda$ is a finite measure on $(0, \infty)$. In many important cases, $\gamma_\lambda$ is given explicitly by~\eqref{eq:gamma} and~\eqref{eq:gamma0}; however, in this section these identities are not used. We extend the definition of $F_\lambda(x)$ by letting $F_\lambda(x) = 0$ for $x \le 0$.

It was proved in Proposition~4.22 in~\cite{bib:mk10} that
\formula[eq:laplaceest0]{
 \laplace F_\lambda(\xi) & \le \frac{\lambda + \xi}{\lambda^2 + \xi^2} \le \frac{2}{\lambda + \xi} \/ , && \lambda, \xi > 0 .
}
By combining \eqref{eq:lf1} and Proposition~\ref{prop:daggerest}, we obtain the following more detailed estimate.

\begin{corollary}
\label{cor:laplaceest}
We have
\formula[eq:laplaceest]{
 e^{-2 \catalan / \pi} \, \frac{\lambda}{\lambda^2 + \xi^2} \, \sqrt{\frac{\psi_\lambda(\xi^2)}{\psi_\lambda(\lambda^2)}} & \le \laplace F_\lambda(\xi) \le e^{2 \catalan / \pi} \, \frac{\lambda}{\lambda^2 + \xi^2} \, \sqrt{\frac{\psi_\lambda(\xi^2)}{\psi_\lambda(\lambda^2)}}
}
for $\lambda, \xi > 0$.\qed
\end{corollary}

H{\"o}lder continuity of $F_\lambda$ was already studied in Lemma~4.24 in~\cite{bib:mk10}. However, when $\psi$ does not have a power-type growth at infinity, $F_\lambda$ fails to be H{\"o}lder continuous. On the other hand, uniform estimate for $F_\lambda(x)$ for small $x > 0$ is crucial in the proof of Theorem~\ref{th:fpt}. We generalize the estimates of~\cite{bib:mk10} in the two lemmas below.

Since $G_\lambda(x)$ is completely monotone, we have $G_\lambda'(x) \le 0$ and $G_\lambda''(x) \ge 0$ for $\lambda, x > 0$. It follows that the function $F_\lambda(x) = \sin(\lambda x + \thet_\lambda) - G_\lambda(x)$ is increasing on $[0, (\pi/2 - \thet_\lambda) / \lambda]$ and concave on $[0, (\pi - \thet_\lambda) / \lambda]$ (this explains the role of upper bounds for $\thet_\lambda$). Furthermore, $0 \le G_\lambda(x) \le \sin \thet_\lambda$ (Lemma~4.21 in~\cite{bib:mk10}), and so $-1 - \sin \thet_\lambda \le F_\lambda(x) \le 1$.

\begin{lemma}
\label{lem:flambdaest}
When $\lambda > 0$ and $0 < \lambda x \le (\pi/2 - \thet_\lambda) / 2$, we have
\formula[eq:flambdaest]{
\begin{aligned}
 \frac{e^{1 - 2 \catalan / \pi}}{2(e + 1)} \, \lambda x \sqrt{\frac{\psi_\lambda(1/x^2)}{\psi_\lambda(\lambda^2)}} & \le F_\lambda(x) \\
 & \hspace*{-3em} \le \expr{2 e^{1 + 2 \catalan / \pi} + \frac{4 e}{\pi/2 - \thet_\lambda}} \lambda x \, \sqrt{\frac{\psi_\lambda(1/x^2)}{\psi_\lambda(\lambda^2)}} \/ .
\end{aligned}
}
\end{lemma}

\begin{proof}
Let $\lambda > 0$, $b = (\pi/2 - \thet_\lambda) / \lambda$ and $0 < a \le b$. By concavity of the sine function and convexity of $G_\lambda$, for any $x > a$ we have
\formula{
 \frac{\sin(\lambda x + \thet_\lambda) - \sin \thet_\lambda}{x} & \le \frac{\sin(\lambda a + \thet_\lambda) - \sin \thet_\lambda}{a} \/ , \\
 \frac{\sin \thet_\lambda - G_\lambda(x)}{x} & \le \frac{\sin \thet_\lambda - G_\lambda(a)}{a} \/ .
}
Hence $F_\lambda(x) \le (x/a) F_\lambda(a)$ when $x > a$. This and monotonicity of $F_\lambda$ on $[0, b]$ yield
\formula{
 F_\lambda(x) & \le F_\lambda(a) \max(1, x/a)
}
for all $x > 0$. By Proposition~\ref{prop:lb}, for $\xi > 0$ we have
\formula{
 F_\lambda(a) & \ge \frac{\xi \laplace F_\lambda(\xi)}{1 + (a \xi)^{-1} e^{-a \xi}} .
}
For $\xi = 1/a$, this gives
\formula{
 F_\lambda(a) & \ge \frac{e}{e + 1} \, \frac{\laplace F_\lambda(1/a)}{a} \, .
}
On the other hand, $F_\lambda(x)$ is increasing in $x \in (0, b)$, and clearly $F_\lambda(x) \ge -2$ for all $x \ge b$. Proposition~\ref{prop:ub2} gives
\formula{
 F_\lambda(a) & \le \frac{e^{a \xi} \xi \laplace F_\lambda(\xi) + 2 e^{-(\pi/2 - \thet_\lambda - \lambda a) \xi / \lambda}}{1 - e^{-(\pi/2 - \thet_\lambda - \lambda a) \xi / \lambda}}
}
for all $\xi > 0$. Let $\xi = 1/a$. Using the inequality $1 - e^{-s} \ge s/(s+1)$ ($s > 0$) in the denominator and $e^{-1/s} \le s$ ($s > 0$) in the numerator, we obtain that
\formula{
 F_\lambda(a) & \le \frac{1}{1 - e^{-(\pi/2 - \thet_\lambda - \lambda a) / (\lambda a)}} \expr{\frac{e \laplace F_\lambda(1/a)}{a} + 2 e^{-(\pi/2 - \thet_\lambda - \lambda a) / (\lambda a)}} \\
 & = \frac{e}{1 - e^{-(\pi/2 - \thet_\lambda - \lambda a) / (\lambda a)}} \expr{\frac{\laplace F_\lambda(1/a)}{a} + 2 e^{-(\pi/2 - \thet_\lambda) / (\lambda a)}} \\
 & \le \frac{e (\pi/2 - \thet_\lambda)}{\pi/2 - \thet_\lambda - \lambda a} \expr{\frac{\laplace F_\lambda(1/a)}{a} + \frac{2 \lambda a}{\pi/2 - \thet_\lambda}} \/ .
}
The above bounds for $F_\lambda(a)$ combined with~\eqref{eq:laplaceest} give
\formula{
 \frac{e^{1 - 2 \catalan / \pi}}{e + 1} \, \frac{\lambda / a}{\lambda^2 + 1/a^2} \, \sqrt{\frac{\psi_\lambda(1/a^2)}{\psi_\lambda(\lambda^2)}} & \le F_\lambda(a) \\
 & \hspace*{-5em} \le \frac{e (\pi/2 - \thet_\lambda)}{\pi/2 - \thet_\lambda - \lambda a} \expr{e^{2 \catalan / \pi} \, \frac{\lambda / a}{\lambda^2 + 1/a^2} \, \sqrt{\frac{\psi_\lambda(1/a^2)}{\psi_\lambda(\lambda^2)}} + \frac{2 \lambda a}{\pi/2 - \thet_\lambda}} .
}
Assume now that $a \in (0, b/2]$. Then $a \le b/2 = (\pi/2 - \thet_\lambda) / (2 \lambda) \le 1 / \lambda$, and therefore $1 / a^2 \le \lambda^2 + 1/a^2 \le 2 / a^2$. It follows that
\formula{
 \frac{e^{1 - 2 \catalan / \pi} \lambda a}{2(e + 1)} \, \sqrt{\frac{\psi_\lambda(1/a^2)}{\psi_\lambda(\lambda^2)}} & \le F_\lambda(a) \\
 & \hspace*{-5em} \le \frac{e (\pi/2 - \thet_\lambda)}{\pi/2 - \thet_\lambda - \lambda a} \expr{e^{2 \catalan / \pi} \lambda a \, \sqrt{\frac{\psi_\lambda(1/a^2)}{\psi_\lambda(\lambda^2)}} + \frac{2 \lambda a}{\pi/2 - \thet_\lambda}} .
}
Since $\psi_\lambda(\xi)$ is increasing on $(0, \infty)$ and $\lambda < 1 / a$, we have $1 \le \sqrt{\psi_\lambda(1/a^2) / \psi_\lambda(\lambda^2)}$. Finally, $\lambda a \le \lambda b / 2 = (\pi/2 - \thet_\lambda)/2$, so that $(\pi/2 - \thet_\lambda) / (\pi/2 - \thet_\lambda - \lambda a) \le 2$. Formula~\eqref{eq:flambdaest} follows.
\end{proof}

Since $F_\lambda(x) = \sin(\lambda x + \thet_\lambda) - G_\lambda(x)$ for a completely monotone $G_\lambda$, the modulus of continuity of $F_\lambda$ is described by the behavior of $F_\lambda(x)$ for small $x$. More precisely, we have the following result.

\begin{lemma} 
\label{lem:flambdacontinuity}
We have
\formula[eq:flambdacontinuity]{
 |F_\lambda(x) - F_\lambda(y)| & \le \min \expr{\frac{30 \lambda |x - y|}{\pi/2 - \thet_\lambda} \, \sqrt{\frac{\psi_\lambda(1/|x-y|^2)}{\psi_\lambda(\lambda^2)}} \, , \, 4}
}
for all $\lambda > 0$ and $x, y \ge 0$.
\end{lemma}

\begin{proof}
The inequality $|F_\lambda(x) - F_\lambda(y)| \le 4$ is clear. For $x, y > 0$ and $\lambda > 0$,
\formula[eq:flct]{
\begin{aligned}
 |F_\lambda(x) - F_\lambda(y)| & \le |G_\lambda(x) - G_\lambda(y)| + \lambda |x - y| \\
 & \le |G_\lambda(|x - y|) - G_\lambda(0^+)| + \lambda |x - y| \\
 &\le |F_\lambda(|x - y|)| + 2 \lambda |x - y| .
\end{aligned}
}
The same inequality is obviously true also when $x = 0$ or $y = 0$, so from now on we assume that $x, y \ge 0$.

Suppose that $\lambda |x - y| \le (\pi/2 - \thet_\lambda) / 2$. Then, by~\eqref{eq:flct} and \eqref{eq:flambdaest},
\formula{
 |F_\lambda(x) - F_\lambda(y)| & \le \expr{2 e^{1 + 2 \catalan / \pi} + \frac{4 e}{\pi/2 - \thet_\lambda}} \lambda |x - y| \, \sqrt{\frac{\psi_\lambda(1/|x - y|^2)}{\psi_\lambda(\lambda^2)}} + 2 \lambda |x - y| .
}
As in the proof of Lemma~\ref{lem:flambdaest}, $1 \le \sqrt{\psi_\lambda(1/|x-y|^2) / \psi_\lambda(\lambda^2)}$ (indeed $\lambda < 1 / |x - y|$, and $\psi_\lambda$ is increasing), so that
\formula{
 |F_\lambda(x) - F_\lambda(y)| & \le \expr{2 e^{1 + 2 \catalan / \pi} + \frac{4 e}{\pi/2 - \thet_\lambda} + 2} \lambda |x - y| \, \sqrt{\frac{\psi_\lambda(1/|x - y|^2)}{\psi_\lambda(\lambda^2)}} .
}
Finally, the parenthesised expression does not exceed $(\pi e^{1 + 2 \catalan / \pi} + 4 e + \pi) / (\pi/2 - \thet_\lambda)$, and~\eqref{eq:flambdacontinuity} follows.

It remains to consider $\lambda |x - y| > (\pi/2 - \thet_\lambda) / 2$. We claim that in this case, the minimum in~\eqref{eq:flambdacontinuity} is equal to $4$. By Proposition~\ref{prop:cbf:prop}(a), $f(\xi) = \xi / \psi_\lambda(\xi)$ is a CBF. Hence, $f(x) / f(a) \ge \min(1, x / a)$. We obtain
\formula{
 \frac{\lambda |x - y|}{\pi/2 - \thet_\lambda} \, \sqrt{\frac{\psi_\lambda(1/|x-y|^2)}{\psi_\lambda(\lambda^2)}} & = \frac{1}{\pi/2 - \thet_\lambda} \, \sqrt{\frac{\lambda^2 / \psi_\lambda(\lambda^2)}{(1/|x-y|^2) / \psi_\lambda(1/|x-y|^2)}} \\
 & \ge \frac{\min(\lambda |x - y|, 1)}{\pi/2 - \thet_\lambda} > \frac{1}{2} \/ .
}
This proves our claim, and the proof is complete.
\end{proof}

The behavior of $F_\lambda(x)$ when $x \to 0^+$ (with fixed $\lambda > 0$) was studied in Lemma~4.27 in~\cite{bib:mk10}: if $\psi$ is an unbounded CBF regularly varying of order $\alpha \in [0, 1]$ at infinity, then $F_\lambda(x)$ is regularly varying of order $\alpha$ at $0$, and
\formula[eq:flambdaregular]{
 \lim_{x \to 0^+} \expr{\sqrt{\psi(1/x^2)} \, F_\lambda(x)} & = \frac{\sqrt{\lambda^2 \psi'(\lambda^2)}}{\Gamma(1 + \alpha)} \/ , && \lambda > 0 .
}
The next result shows that as $\lambda \to 0^+$ with fixed $x > 0$, $F_\lambda(x)$ behaves as $\lambda \sqrt{\psi'(\lambda^2)} V(x)$, where $V(x)$ is the renewal function of the ascending ladder-height process, $\laplace V(\xi) = \xi / \psi^\dagger(\xi)$ (see Preliminaries).

\begin{proposition}
\label{prop:flambdalimit}
Suppose that $\psi$ is an unbounded CBF. The function $F_\lambda(x)$ is jointly continuous in $\lambda > 0$, $x \ge 0$. Furthermore, if $\limsup_{\lambda \to 0^+} \thet_\lambda < \pi/2$, then
\formula{
 \lim_{\lambda \to 0^+} \frac{F_\lambda(x)}{\lambda \sqrt{\psi'(\lambda^2)}} & = V(x) , && x \ge 0 ,
}
and the convergence is locally uniform in $x \ge 0$. In other words, $F_\lambda(x) / (\lambda \sqrt{\psi'(\lambda^2)})$ extends to a continuous function in $\lambda, x \ge 0$.
\end{proposition}

\begin{proof}
Roughly speaking, we use estimates of $\laplace F_\lambda$ to show $L^2$ continuity of $e^{-x} F_\lambda(x)$ in $\lambda > 0$, and then estimates of $F_\lambda$ to replace $L^2$ convergence by locally uniform convergence.

Let $\lambda_n \to \lambda > 0$. By Corollary~\ref{cor:psilambdacontinuity}, $\laplace F_{\lambda_n}(1 + i \xi)$ converges to $\laplace F_\lambda(1 + i \xi)$ pointwise for $\xi \in \R$. By~\eqref{eq:laplaceest0} and dominated convergence, $\laplace F_{\lambda_n}(1 + i \xi)$ converges to $\laplace F_\lambda(1 + i \xi)$ in $L^2(\R)$. Note that $\laplace F_\lambda(1 + i \xi)$ is the Fourier transform of $e^{-x} F_\lambda(x)$. By Plancherel's theorem, $e^{-x} F_{\lambda_n}(x)$ converges to $e^{-x} F_\lambda(x)$ in $L^2(\R)$. By~\eqref{eq:flambdacontinuity}, the sequence $F_{\lambda_n}(x)$ is equicontinuous in $x$ (here we use the assumption that $\psi$ is unbounded), and hence it converges to $F_\lambda(x)$ locally uniformly in $x \ge 0$. The first part of the proposition is proved.

As $\lambda \to 0^+$, the functions $(\lambda^2 / \psi(\lambda^2)) \psi_\lambda(\xi)$ converge pointwise to the CBF $\xi / \psi(\xi)$. Therefore, by Proposition~\ref{prop:daggercontinuity}, $\sqrt{\lambda^2 / \psi(\lambda^2)} \, \psi_\lambda^\dagger(\xi)$ converges to $\xi / \psi^\dagger(\xi)$ ($\xi \in \C \setminus (-\infty, 0]$). We conclude that $\laplace F_\lambda(1 + i \xi) / (\lambda \sqrt{\psi'(\lambda^2)})$ converges to $\laplace V(1 + i \xi)$ ($\xi \in \R$). Again, we obtain $L^2(\R)$ convergence of $e^{-x} F_\lambda(x) / (\lambda \sqrt{\psi'(\lambda^2)})$ to $e^{-x} V(x)$. The proof will be complete if we show that the family of functions $F_\lambda(x) / (\lambda \sqrt{\psi'(\lambda^2)})$ is equicontinuous as $\lambda \to 0^+$.

By~\eqref{eq:flambdacontinuity}, we have
\formula[eq:flambdaest0]{
\begin{split}
 \frac{|F_\lambda(x) - F_\lambda(y)|}{\lambda \sqrt{\psi'(\lambda^2)}} & \le \frac{30 |x - y|}{\pi/2 - \thet_\lambda} \, \sqrt{\frac{\psi_\lambda(1/|x-y|^2)}{\psi'(\lambda^2) \psi_\lambda(\lambda^2)}} \\
 & = \frac{30}{\pi/2 - \thet_\lambda} \, \sqrt{\frac{1 - \lambda^2 |x-y|^2}{\psi(1/|x-y|^2) - \psi(\lambda^2)}} \/ .
\end{split}
}
Fix $\lambda_0 > 0$ small enough. By the assumption, $\pi/2 - \thet_\lambda$ is bounded below by a positive constant for $\lambda \in (0, \lambda_0)$. It follows that for $\lambda \in (0, \lambda_0)$ and $x, y \ge 0$ satisfying $|x - y| < 1 / \lambda_0$, the expression on the right hand side of~\eqref{eq:flambdaest0} is bounded above by $c / \sqrt{\psi(1/|x - y|)^2 - \psi(\lambda_0^2)}$ (with $c$ depending only on $\psi$ and $\lambda_0$). This upper bound does not depend on $\lambda \in (0, \lambda_0)$, and since $\psi$ is unbounded, it converges to $0$ as $|x - y| \to 0$.
\end{proof}

%
%

\section{Suprema and first passage times for complete Bernstein functions}
\label{sec:supcbf}

Below we prove Theorems~\ref{th:fpt}--\ref{th:fptdensityasymp}. We remark that the condition~\eqref{eq:fpt:a1} is used only to assert that $\sup_{\lambda > 0} \thet_\lambda < \pi/2$, and can be replaced by the latter condition.

\begin{proof}[Proof of Theorem~\ref{th:fpt}]
First we consider $n = 0$, that is, we will show that
\formula[eq:fpt0]{
 \pr(\tau_x > t) & = \frac{2}{\pi} \int_0^\infty \sqrt{\frac{\psi'(\lambda^2)}{\psi(\lambda^2)}} \, e^{-t \psi(\lambda^2)} F_\lambda(x) d\lambda .
}
We claim that the integral in~\eqref{eq:fpt0} converges, and that $e^{-\xi x} e^{-t \psi(\lambda^2)} F_\lambda(x) \sqrt{\psi'(\lambda^2) / \psi(\lambda^2)}$ is jointly integrable in $\lambda, x > 0$ for any $\xi > 0$. Assuming this is true, the proof is quite straightforward. Indeed, by Fubini, the Laplace transform (in $x$) of the right hand side of~\eqref{eq:fpt0} is then
\formula{
 \frac{2}{\pi} \int_0^\infty \sqrt{\frac{\psi'(\lambda^2)}{\psi(\lambda^2)}} \, e^{-t\psi(\lambda^2)} \laplace F_\lambda(\xi) d\lambda .
}
By Theorem~\ref{th:suplaplace} (see~\eqref{eq:ltau}) and Theorem~\ref{th:eigenfunctions} (see~\eqref{eq:lf0}), this is equal to the Laplace transform (in $x$) of $\pr(\tau_x > t)$, and the result follows by the uniqueness of the Laplace transform. Hence it is enough to prove our claim.

Let $t \ge t_0$, $\Theta = \sup_{\lambda > 0} \thet_\lambda$ and $\lambda_0 = (\pi/2 - \Theta) / (2 x)$. By the assumption~\eqref{eq:fpt:a1} and Proposition~\ref{prop:thetaest2}, $\Theta < \pi/2$, and hence $\lambda_0 > 0$. Note that $\lambda_0 x \le \pi/4 < 1$. Since $|F_\lambda(x)| \le 2$ and $\sqrt{\psi'(\lambda^2) / \psi(\lambda^2)} \le 1/\lambda$, we have
\formula{
 \hspace*{3em} & \hspace*{-3em} \int_{\lambda_0}^\infty \sqrt{\frac{\psi'(\lambda^2)}{\psi(\lambda^2)}} \, e^{-t\psi(\lambda^2)} |F_\lambda(x)| d\lambda \\
 & \le 2 \int_{\min(\lambda_0, 1)}^1 \frac{e^{-t\psi(\lambda^2)}}{\lambda} \, d\lambda + 2 \int_1^\infty \sqrt{\frac{\psi'(\lambda^2)}{\psi(\lambda^2)}} \, e^{-t\psi(\lambda^2)} d\lambda \\
 & \le 2 \max(0, -\log \lambda_0) + 2 \int_1^\infty \sqrt{\frac{\psi'(\lambda^2)}{\psi(\lambda^2)}} \, e^{-t_0 \psi(\lambda^2)} d\lambda \/ ,
}
which is finite by assumption~\eqref{eq:fpt:a2}. We now consider $\lambda \in (0, \lambda_0)$. By Lemma~\ref{lem:flambdaest}, we have
\formula{
 0 \le \sqrt{\frac{\psi'(\lambda^2)}{\psi(\lambda^2)}} \, e^{-t\psi(\lambda^2)} F_\lambda(x) & \le c_1(\Theta) \lambda x \, \sqrt{\frac{\psi'(\lambda^2) \psi_\lambda(1/x^2)}{\psi(\lambda^2) \psi_\lambda(\lambda^2)}} \, .
}
Furthermore, by~\eqref{eq:psilambdaest3} (recall that $\lambda < \lambda_0 < 1 / x$),
\formula{
 \lambda x \, \sqrt{\frac{\psi'(\lambda^2) \psi_\lambda(1/x^2)}{\psi(\lambda^2) \psi_\lambda(\lambda^2)}} & \le \frac{\lambda \psi'(\lambda^2)}{\sqrt{\psi(\lambda^2) (\psi(1/x^2) - \psi(\lambda^2))}} \, .
}
Integration and substitution $z = \psi(\lambda^2)$ give
\formula{
 & \int_0^{\lambda_0} \sqrt{\frac{\psi'(\lambda^2)}{\psi(\lambda^2)}} \, e^{-t\psi(\lambda^2)} |F_\lambda(x)| d\lambda \le c_1(\Theta) \int_0^{\lambda_0} \frac{\lambda \psi'(\lambda^2)}{\sqrt{\psi(\lambda^2) (\psi(1 / x^2) - \psi(\lambda^2))}} \, d\lambda \\
 & \qquad = \frac{c_1(\Theta)}{2} \int_0^{\psi(\lambda_0^2)} \frac{1}{\sqrt{z (\psi(1/x^2) - z)}} \, dz \le \frac{c_1(\Theta) \pi}{2} \, .
}
We conclude that
\formula{
 \int_0^\infty \sqrt{\frac{\psi'(\lambda^2)}{\psi(\lambda^2)}} \, e^{-t\psi(\lambda^2)} |F_\lambda(x)| d\lambda & \le c_2(\psi, t) + 2 \max(0, \log x) ,
}
which shows that the integral in~\eqref{eq:fpt0} is absolutely convergent. This also shows that
\formula{
 \int_0^\infty \int_0^\infty e^{-\xi x} \sqrt{\frac{\psi'(\lambda^2)}{\psi(\lambda^2)}} \, e^{-t\psi(\lambda^2)} |F_\lambda(x)| d\lambda d\xi & < \infty ,
}
as desired. The proof in the case $n = 0$ is complete.

For general $n \ge 0$, one uses dominated convergence to prove that the derivative can by taken under the integral sign; we omit the details.
\end{proof}

We now find upper and lower bounds for the derivatives in $t$ of $\pr(\tau_x > t)$. Note that the estimates of $\pr(\tau_x > t)$ are covered by~\cite{bib:kmr11} for a wider class of L{\'e}vy processes. We begin with a simple technical result.

\begin{proposition}
\begin{enumerate}
\item[(a)] Let $\gamma(a; x) = \int_0^x e^{-s} s^{a-1} ds$ ($a, x > 0$) be the lower incomplete gamma function. Then
\formula[eq:ligest]{
 \frac{\min(1, x^a)}{a e} \le \gamma(a; x) & \le \frac{\min(\Gamma(a+1), x^a)}{a} \, , && a, x > 0 .
}
\item[(b)] We have
\formula[eq:auxintest]{
 \int_0^x \frac{e^{-s} s^{a - 1}}{\sqrt{x - s}} \, ds & \le \frac{\min(2 \Gamma(a) + \pi a^a e^{-a}, 2 (1/a + 1) x^a)}{\sqrt{2 x}} \, , && a, x > 0 .
}
\end{enumerate}
\end{proposition}

\begin{proof}
For the part~(a), we simply have
\formula{
 \gamma(a; x) & = \int_0^x e^{-s} s^{a-1} ds \ge \frac{1}{e} \int_0^{\min(1, x)} s^{a-1} ds = \frac{\min(1, x^a)}{a e} \, ,
}
and
\formula{
 \gamma(a; x) & = \int_0^x e^{-s} s^{a-1} ds \le \min\expr{\Gamma(a), \int_0^x s^{a-1} ds} = \min\expr{\Gamma(a), \frac{x^a}{a}} .
}
To prove~(b), we split the integral into two parts. First,
\formula{
 \int_0^{x/2} \frac{e^{-s} s^{a - 1}}{\sqrt{x - s}} \, ds & \le \sqrt{\frac{2}{x}} \, \int_0^{x/2} e^{-s} s^{a - 1} ds \\
 & = \sqrt{\frac{2}{x}} \, \gamma(a; x/2) \le \sqrt{\frac{2}{x}} \, \min\expr{\Gamma(a), \frac{x^a}{a}} .
}
Next, $s^a e^{-s}$ attains its maximum at $s = a$. Hence,
\formula{
 \int_{x/2}^x \frac{e^{-s} s^{a - 1}}{\sqrt{x - s}} \, ds & \le a^a e^{-a} \int_{x/2}^x \frac{1}{\sqrt{s^2 (x - s)}} \, ds \\
 & \le a^a e^{-a} \, \sqrt{\frac{2}{x}} \, \int_{x/2}^x \frac{1}{\sqrt{s (x - s)}} \, ds = \frac{\pi a^a e^{-a}}{2} \, \sqrt{\frac{2}{x}} \, .
}
Finally,
\formula{
 \int_{x/2}^x \frac{e^{-s} s^{a - 1}}{\sqrt{x - s}} \, ds & \le x^{a - 1} \int_{x/2}^x \frac{1}{\sqrt{x - s}} \, ds = \sqrt{2} \, x^{a - 1/2} .
}
It follows that
\formula{
 \int_0^x \frac{e^{-s} s^{a - 1}}{\sqrt{x - s}} \, ds & \le \sqrt{\frac{2}{x}} \, \frac{\min(\Gamma(a+1), x^a)}{a} + \min\expr{\frac{\pi a^a e^{-a}}{2} \, \sqrt{\frac{2}{x}}, \sqrt{2} \, x^{a - 1/2}} \\
 & \le \sqrt{\frac{2}{x}} \expr{\frac{\min(\Gamma(a+1), x^a)}{a} + \min\expr{\frac{\pi a^a e^{-a}}{2}, x^a}} ,
}
which gives~\eqref{eq:auxintest}.
\end{proof}

\begin{lemma}
\label{lem:fptdensityest}
Suppose that~\eqref{eq:fpt:a1} and~\eqref{eq:fpt:a2} hold for some $t_0 > 0$. For $n \ge 0$, $t > t_0$ ($t \ge t_0$ if $n = 0$) and  $\lambda_0 > 0$, denote (cf~\eqref{eq:fpt:a2})
\formula{
 I_n(t_0, \lambda_0) & = \frac{4}{\pi} \int_{\lambda_0}^\infty e^{-t_0 \psi(\lambda^2)} \sqrt{\psi'(\lambda^2)} (\psi(\lambda^2))^{n-1/2} d\lambda .
}
Let $\Theta = \sup_{\xi > 0} \thet_\lambda$ and $\lambda_0(x) = (\pi/2 - \Theta) / (2 x)$. If $t > t_0$ ($t \ge t_0$ if $n = 0$) and $x > 0$, then
\formula[]{
\label{eq:fptdensityest}
\begin{split}
 & C_1(n, \Theta) \, \min\expr{(\psi(1/x^2))^n, \frac{1}{t^{n + 1/2} \sqrt{\psi(1/x^2)}}} + I_n(t, \lambda_0(x)) \\
 & \qquad \le (-1)^n \, \frac{d^n}{dt^n} \, \pr(\tau_x > t) \\ & \qquad \le C_2(n, \Theta) \min\expr{(\psi(1/x^2))^n, \frac{1}{t^{n + 1/2} \sqrt{\psi(1/x^2)}}} - I_n(t, \lambda_0(x)) .
\end{split}
}
\end{lemma}

\begin{remark}
\label{rem:remainder}
\begin{enumerate}
\item
When~\eqref{eq:fpt:a2} holds for some $t_0 > 0$, then clearly $I_0(t_0, \lambda_0) < \infty$ for any $\lambda_0 > 0$. Furthermore, for any $n > 0$, $t > t_0$ and $\lambda_0 > 0$, we have $e^{-(t - t_0) \psi(\lambda^2)} \le c(n, \lambda_0) (\psi(\lambda^2))^n$ for $\lambda \ge \lambda_0$. Hence, \eqref{eq:fpt:a2} implies also that $I_n(t, \lambda_0) < \infty$ for any $n > 0$, $t > t_0$ and $\lambda_0 > 0$.
\item
When $t > t_0$, then $I_n(t, \lambda_0) \le e^{-(t-t_0) \psi(\lambda_0^2)} I_n(t_0, \lambda_0)$. Hence, $I_n(t, \lambda_0)$, if finite, decays exponentially fast as $t \to \infty$. Therefore, $I_n(t, \lambda_0)$ can be considered as the `error term'; see Theorem~\ref{th:fptdensityasymp}.\qed
\end{enumerate}
\end{remark}

\begin{proof}[Proof of Lemma~\ref{lem:fptdensityest}]
By the assumption~\eqref{eq:fpt:a1} and Proposition~\ref{prop:thetaest2}, $\Theta < \pi/2$. Let $k = (\pi/2 - \Theta) / 2 \in (0, \pi/4)$, so that $\lambda_0 = \lambda_0(x) = k / x$. Denote the integrand in~\eqref{eq:fpt} by $f_{n,t,x}(\lambda)$,
\formula{
 f_{n,t,x}(\lambda) & = \sqrt{\frac{\psi'(\lambda^2)}{\psi(\lambda^2)}} \, (\psi(\lambda^2))^n e^{-t \psi(\lambda^2)} F_\lambda(x) .
}
By Lemma~\ref{lem:flambdaest}, for $\lambda \in (0, \lambda_0)$ we have
\formula{
 f_{n,t,x}(\lambda) & \ge (\psi(\lambda^2))^{n - 1/2} \sqrt{\psi'(\lambda^2)} \, e^{-t\psi(\lambda^2)} \, \frac{e^{1 - 2 \catalan / \pi} \lambda x}{2(e + 1)} \, \sqrt{\frac{\psi_\lambda(1/x^2)}{\psi_\lambda(\lambda^2)}} \\
 & = \frac{e^{1 - 2 \catalan / \pi} \lambda \psi'(\lambda^2)}{2(e + 1)} \, e^{-t\psi(\lambda^2)} \frac{(\psi(\lambda^2))^{n - 1/2} \sqrt{1 - \lambda^2 x^2}}{\sqrt{\psi(1/x^2) - \psi(\lambda^2)}} \\
 & \ge \frac{e^{1 - 2 \catalan / \pi} \lambda \psi'(\lambda^2)}{2(e + 1)} \, e^{-t\psi(\lambda^2)} \frac{(\psi(\lambda^2))^{n - 1/2} \sqrt{1 - k^2}}{\sqrt{\psi(1/x^2))}} \, .
}
Since $\lambda x < k$ and $\psi(\lambda_0^2) = \psi(k^2 / x^2) \ge k^2 \psi(1 / x^2)$ ($\psi$ is nonnegative and concave, and $k < 1$), we obtain (with $z = \psi(\lambda^2)$)
\formula{
 \int_0^{\lambda_0} f_{n,t,x}(\lambda) d\lambda & \ge \frac{e^{1 - 2 \catalan / \pi} \sqrt{1 - k^2}}{4 (e + 1) \sqrt{\psi(1/x^2)}} \int_0^{\psi(\lambda_0^2)} z^{n - 1/2} e^{-t z} dz \\
 & \hspace*{-2em} \ge \frac{e^{1 - 2 \catalan / \pi} \sqrt{1 - k^2}}{4 (e + 1) t^{n + 1/2} \sqrt{\psi(1/x^2)}} \, \gamma(n + 1/2; k^2 t \psi(1/x^2)) ,
}
where $\gamma$ is the lower incomplete gamma function. By~\eqref{eq:ligest},
\formula[eq:fptdensityest1]{
\begin{split}
 \int_0^{\lambda_0} f_{n,t,x}(\lambda) d\lambda & \ge \frac{e^{1 - 2 \catalan / \pi} \sqrt{1 - k^2}}{4 (e + 1) t^{n + 1/2} \sqrt{\psi(1/x^2)}} \, \frac{\min(1, (k^2 t \psi(1/x^2))^{n + 1/2})}{(n + 1/2) e} \, .
\end{split}
}
In a similar way,
\formula{
 f_{n,t,x}(\lambda) & \le \expr{2 e^{1 + 2 \catalan / \pi} + \frac{2 e}{k}} \lambda \psi'(\lambda^2) e^{-t\psi(\lambda^2)} \, \frac{(\psi(\lambda^2))^{n - 1/2}}{\sqrt{\psi(1/x^2) - \psi(\lambda^2)}} \/ ,
}
and so
\formula{
 \int_0^{\lambda_0} f_{n,t,x}(\lambda) d\lambda & \le \expr{e^{1 + 2 \catalan / \pi} + \frac{e}{k}} \int_0^{\psi(1/x^2)} \frac{z^{n - 1/2} e^{-t z}}{\sqrt{\psi(1/x^2) - z}} \, dz \\
 & = \expr{e^{1 + 2 \catalan / \pi} + \frac{e}{k}} \frac{1}{t^n} \, \int_0^{t \psi(1/x^2)} \frac{s^{n - 1/2} e^{-s}}{\sqrt{t \psi(1/x^2) - s}} \, ds .
}
By~\eqref{eq:auxintest}, with $c_1(n) = 2 \Gamma(n + 1/2) + \pi ((n + 1/2) / e)^{n + 1/2}$ and $c_2(n) = 2 (1/(n + 1/2) + 1)$,
\formula[eq:fptdensityest2]{
 \int_0^{\lambda_0} f_{n,t,x}(\lambda) d\lambda & \le \expr{e^{1 + 2 \catalan / \pi} + \frac{e}{k}} \frac{\min(c_1(n), c_2(n) (t \psi(1/x^2))^{n + 1/2})}{t^n \sqrt{2 t \psi(1/x^2)}} .
}
Hence, we found a two-sided estimate for the integral of $f_{n,t,x}(\lambda)$ over $(0, \lambda_0)$. The integral over $(\lambda_0, \infty)$ is highly oscillatory, and therefore difficult to estimate. For this reason, we are satisfied with a simple bound obtained using the inequality $|F_\lambda(x)| \le 2$,
\formula[eq:fptdensityest3]{
\begin{aligned}
 \frac{2}{\pi} \int_{\lambda_0}^\infty |f_{n,t,x}(\lambda)| d\lambda & \le \frac{4}{\pi} \int_{\lambda_0}^\infty (\psi(\lambda^2))^{n - 1/2} \sqrt{\psi'(\lambda^2)} e^{-t \psi(\lambda^2)} d\lambda \\ & = I_n(t, \lambda_0) .
\end{aligned}
}
The lower bound in~\eqref{eq:fptdensityest} is a consequence of~\eqref{eq:fptdensityest1} and \eqref{eq:fptdensityest3}, and the upper bound in~\eqref{eq:fptdensityest} follows from~\eqref{eq:fptdensityest2} and~\eqref{eq:fptdensityest3}.
\end{proof}

We remark that in the statement of the lemma, we can take
\formula{
 & C_1(n, \Theta) = \frac{e^{1 - 2 \catalan / \pi} \sqrt{1 - \pi^2 / 16} \, (\pi/4 - \Theta/2)^{2n + 1}}{2 \pi e (e + 1) (n + 1/2)} \ge \frac{(\pi/2 - \Theta)^{2n + 1}}{17 (n + 1/2) 2^{2n + 3}} \, , \\
 & C_2(n, \Theta) = \frac{e \sqrt{2}}{\pi} \expr{e^{2 \catalan / \pi} + \frac{2}{\pi/2 - \Theta}} \expr{2 \Gamma(n + 1/2) + \pi ((n + 1/2) / e)^{n + 1/2}} \le \frac{15 \, n!}{\pi/2 - \Theta} \, .
}
Note that $C_1(n, \Theta)$ decreases with $\Theta$, while $C_2(n, \Theta)$ increases with $\Theta$. The notation of Lemma~\ref{lem:fptdensityest}, namely $I_n(t_0, \lambda_0)$, $C_1(n, \Theta)$ and $C_2(n, \Theta)$, is kept in the remaining part of the section.

\begin{corollary}
\label{cor:fptdensityest}
Let $\eps > 0$, $n \ge 1$, $t_0 > 0$, $x_0 > 0$. If the conditions~\eqref{eq:fpt:a1} and~\eqref{eq:fpt:a2} hold true, then there are positive constants $C_3(n, \Theta)$, $C_4(n, \Theta)$, $C_5(\eps, n, \Theta, I_n(t_0, \lambda_0(x_0)))$ (here $\Theta = \sup_{\xi > 0} \thet_\lambda$ and $\lambda_0(x) = (\pi/2 - \Theta) / (2 x)$) such that
\formula[eq:fptdensityasympest]{
 \frac{C_3(n, \Theta)}{t^{n + 1/2} \sqrt{\psi(1/x^2)}} \le (-1)^n \, \frac{d^n}{dt^n} \, \pr(\tau_x > t) & \le \frac{C_4(n, \Theta)}{t^{n + 1/2} \sqrt{\psi(1/x^2)}}
}
when $x \in (0, x_0]$ and
\formula{
 t & \ge \max\expr{(1 + \eps) t_0, \frac{1}{\psi(1/x^2)}, \frac{C_5(\eps, n, \Theta, I_n(t_0, \lambda_0(x_0)))}{(\psi(1/x^2))^{1 + \eps}}} .
}
\end{corollary}

\begin{proof}
This is a combination of Lemma~\ref{lem:fptdensityest} and Remark~\ref{rem:remainder}. As in the proof of Lemma~\ref{lem:fptdensityest}, we let $k = (\pi/2 - \Theta)/2$, so that $\lambda_0 = \lambda_0(x) = k / x$. Recall that since $k < 1$, we have $\psi(k^2/x^2) \ge k^2 \psi(1/x^2)$. It follows that for $n \ge 1$, $\eps > 0$, $x > 0$ and $t > (1 + \eps) t_0$, we have
\formula{
 \frac{I_n(t, \lambda_0)}{I_n(t_0, \lambda_0)} & \le e^{-(t - t_0) \psi(\lambda_0^2)} \le e^{-\eps / (1 + \eps) t \psi(k^2/x^2)} \\
 & \le e^{-k^2 \eps / (1 + \eps) t \psi(1/x^2)} \le \frac{c_2(\eps, n, \Theta)}{(t \psi(1/x^2))^{n + 1/2 + n / \eps}} \, .
}
Fix $A > 0$. When $t (\psi(1/x^2))^{1 + \eps} \ge A^\eps$, we obtain
\formula{
 I_n(t, \lambda_0) & \le \frac{c_2(\eps, n, \Theta)}{A^n t^{n + 1/2} \sqrt{\psi(1/x^2)}} \, I_n(t_0, \lambda_0) .
}
Hence, by~\eqref{eq:fptdensityest}, if $t \psi(1/x^2) \ge 1$, we have (with the constants $C_1(n, \Theta)$ and $C_2(n, \Theta)$ of Lemma~\ref{lem:fptdensityest})
\formula{
 \expr{C_1(n, \Theta) - \frac{c_2(\eps, n, \Theta) I_n(t_0, \lambda_0)}{A^n}} \frac{1}{t^{n + 1/2} \sqrt{\psi(1/x^2)}} & \le (-1)^n \, \frac{d^n}{dt^n} \, \pr(\tau_x > t) \\
 & \hspace{-13.2em} \le \expr{C_2(n, \Theta) + \frac{c_2(\eps, n, \Theta) I_n(t_0, \lambda_0)}{A^n}} \frac{1}{t^{n + 1/2} \sqrt{\psi(1/x^2)}} \/ .
}
When $x \in (0, x_0]$, then $I_n(t_0, \lambda_0(x)) \le I_n(t_0, \lambda_0(x_0))$. Hence, \eqref{eq:fptdensityasympest} holds if $A > (C_5(\eps, n, \Theta, I_n(t_0, \lambda_0(x_0))))^{1/\eps}$ for some $C_5$.
\end{proof}

\begin{proof}[Proof of Theorem~\ref{th:fptdensityest}]
Part~(a) follows directly from Corollary~\ref{cor:fptdensityest}. For part~(b), suppose that $|\psi''(\xi)| / \psi'(\xi) \le \ro / \xi$ for some $\ro \in [0, 1)$ and all $\xi > 0$. By Proposition~\ref{prop:thetaest2}, with the notation of Lemma~\ref{lem:fptdensityest}, we have $\Theta \le \ro \pi / 4$.

Integrating the inequality $-\psi''(\xi) / \psi'(\xi) \le \ro / \xi$, we obtain $\psi'(\xi_1) / \psi'(\xi_2) \le (\xi_2 / \xi_1)^\ro$ if $0 < \xi_1 < \xi_2$. Integrating this again in $\xi_1$ gives $\psi(\xi) / \psi'(\xi) \le \xi / (1 - \ro)$. Hence, for all $t > 0$ and $\lambda_0 > 0$,
\formula{
 I_n(t, \lambda_0) & \le \frac{4}{\pi} \int_{\lambda_0}^\infty e^{-t \psi(\lambda^2)} \, \frac{\lambda \psi'(\lambda^2)}{\sqrt{(1 - \ro) \psi(\lambda^2)}} \, (\psi(\lambda^2))^{n - 1/2} d\lambda \\
 & = \frac{2}{\pi \sqrt{1 - \ro} \, t^n} \int_{\lambda_0}^\infty e^{-t \psi(\lambda^2)} (t \psi(\lambda^2))^{n - 1} (2 \lambda t \psi'(\lambda^2)) d\lambda \\
 & = \frac{2}{\pi \sqrt{1 - \ro} \, t^n} \, \Gamma(n; t \psi(\lambda_0^2)) ,
}
where $\Gamma(a; z)$ is the upper incomplete Gamma function. In particular, $I_n(t, \lambda_0)$ is finite, and~\eqref{eq:fpt:a2} holds true.

Let $k = (\pi/2 - \Theta) / 2 \in (\pi/8, \pi/4)$, and take $\lambda_0 = \lambda_0(x) = k / x$, as in Lemma~\ref{lem:fptdensityest}. Recall that since $k < 1$, we have $\psi(\lambda_0^2) \ge k^2 \psi(1 / x^2)$. Hence,
\formula[eq:inpowerest]{
\begin{aligned}
 I_n(t, \lambda_0) & \le \frac{2}{\pi \sqrt{1 - \ro}} \, \frac{\Gamma(n; k^2 t \psi(1 / x^2))}{t^n} \\
 & = \frac{2}{\pi \sqrt{1 - \ro}} \, \frac{\sqrt{t \psi(1 / x^2)} \, \Gamma(n; (\pi/8)^2 t \psi(1 / x^2))}{t^{n + 1/2} \sqrt{\psi(1 / x^2)}} \, .
\end{aligned}
}
The constant $c_3(n, \ro)$ in the statement of the theorem is so chosen that for all $s \ge c_3(n, \ro)$,
\formula{
 \frac{2}{\pi \sqrt{1 - \ro}} \, \sqrt{s} \, \Gamma(n; (\pi/8)^2 s) \le \frac{\min(C_1(n, \pi/4), C_2(n, \pi/4))}{2} \, ,
}
with the constants $C_1(n, \Theta)$ and $C_2(n, \Theta)$ from Lemma~\ref{lem:fptdensityest}. Then, by Lemma~\ref{lem:fptdensityest}, when $t \psi(1 / x^2) \ge c_3(n, \ro)$ we have
\formula{
 \frac{C_1(n, \pi/4)}{2 t^{n + 1/2} \sqrt{\psi(1/x^2)}} & \le (-1)^n \, \frac{d^n}{dt^n} \, \pr(\tau_x > t) \le \frac{C_2(n, \pi/4)}{2 t^{n + 1/2} \sqrt{\psi(1/x^2)}} \/ ,
}
as desired. Finally, by~\eqref{eq:inpowerest}, we have $I_n(t, \lambda_0) \le c(n, \ro) / (t^{n + 1/2} \sqrt{\psi(1 / x^2)})$, where $c(n, \ro) = 2 \pi^{-1} (1 - \ro)^{-1/2} \sup_{s > 0} (s^{1/2} \Gamma(n; (\pi / 8)^2 s))$. This and Lemma~\ref{lem:fptdensityest} prove that the upper bound in~\eqref{eq:fptpowerest} holds for all $t, x > 0$, but with a constant depending on $\ro$.
\end{proof}

\begin{remark}
\label{rem:fptdensityest}
\begin{enumerate}
\item[(a)] The strict inequality $\ro < 1$ is essential to the proof. The case $\ro = 1$ is much more complicated; see the example in Subsection~\ref{subsec:geomstable}.
\item[(b)] Theorem~\ref{th:fptdensityest}(b) can be easily generalised to the case when $|\psi''(\xi)| / \psi'(\xi) \le \ro / \xi$ only for $\xi \in (0, \eps) \cup (1 / \eps, \infty)$, with $\eps \in (0, 1)$ fixed. In this case the constant $c_3 = c_3(n, \ro, \eps)$ depends also on $\eps$. We omit the details.
\item[(c)] The upper bound in~\eqref{eq:fptpowerest} is certainly not optimal when $t \psi(1 / x^2)$ is small. This is due to essential cancellations in~\eqref{eq:fpt}.\qed
\end{enumerate}
\end{remark}

\begin{proof}[Proof of Theorem~\ref{th:fptdensityasymp}(a)]
We use the notation of Lemma~\ref{lem:fptdensityest}, and take $\lambda_0 = \lambda_0(x) = (\pi/2 - \Theta) / (2 x)$ for a fixed $x > 0$. Let $f_{n,t,x}(\lambda)$ be the integrand in~\eqref{eq:fpt}. We have
\formula{
 & t^{n + 1/2} \int_0^{\lambda_0} f_{n,t,x}(\lambda) d\lambda = \int_0^{\lambda_0} \frac{F_\lambda(x)}{\lambda \sqrt{\psi'(\lambda^2)}} \, t^{n + 1/2}  (\psi(\lambda^2))^{n - 1/2} e^{-t\psi(\lambda^2)} \lambda \psi'(\lambda^2) d\lambda\/.
}
By Proposition~\ref{prop:flambdalimit}, $\lim_{\lambda \to 0} F_\lambda(x) / (\lambda \sqrt{\psi'(\lambda^2)}) = V(x)$, and therefore $F_\lambda(x) / (\lambda \sqrt{\psi'(\lambda^2)})$ extends to a continuous function of $\lambda \in [0, \lambda_0]$. Let
\formula{
 \mu_t(d\lambda) = t^{n + 1/2} (\psi(\lambda^2))^{n - 1/2} e^{-t\psi(\lambda^2)} \lambda \psi'(\lambda^2) \ind_{[0, \lambda_0]}(\lambda) d\lambda .
}
As $t \to \infty$, the density function of $\mu_t$ converges uniformly to $0$ on $[\eps, \lambda_0]$ for every $\eps > 0$. Hence, $\mu_t(d\lambda)$ converges weakly to a point-mass at $0$. Furthermore, by a substitution $z = t \psi(\lambda^2)$,
\formula{
 \|\mu_t\| & = \frac{1}{2} \int_0^{t \psi(\lambda_0^2)} z^{n - 1/2} e^{-z} dz = \frac{\gamma(n + 1/2; t \psi(\lambda_0^2))}{2} \/ ,
}
and hence $\|\mu_t\|$ converges to $\Gamma(n + 1/2) / 2$ as $t \to \infty$. It follows that
\formula{
 \lim_{t \to \infty} \expr{t^{n + 1/2} \int_0^{\lambda_0} f_{n,t,x}(\lambda) d\lambda} & = \frac{\Gamma(n + 1/2)}{2} \, V(x) .
}
Finally, by~\eqref{eq:fptdensityest3} and Remark~\ref{rem:remainder},
\formula{
 \lim_{t \to \infty} \abs{t^{n + 1/2} \, \frac{4}{\pi} \int_{\lambda_0}^\infty f_{n,t,x}(\lambda) d\lambda} & \le \lim_{t \to \infty} \expr{t^{n + 1/2} I_n(t, \lambda_0)} = 0,
}
and so~\eqref{eq:fptdensityasymp} follows by~\eqref{eq:fpt}. The convergence is locally uniform, since the extension of $F_\lambda(x) / (\lambda \sqrt{\psi'(\lambda^2)})$ is jointly continuous in $\lambda \ge 0$ and $x \ge 0$ (Proposition~\ref{prop:flambdalimit}), and $t^{n + 1/2} I_n(t, \lambda_0(x))$ converges to $0$ locally uniformly in $x \ge 0$.
\end{proof}

\begin{remark}
Alternatively, one can deduce formula~\eqref{eq:fptdensityasymp} (without uniformity in $x$) as follows. In~\cite{bib:gn86} it was proved that $\sqrt{t} \, \pr(\tau_x > t)$ converges to $V(x) / \sqrt{\pi}$ as $t \to \infty$. By Theorem~\ref{th:fptdensityest}(a), $\pr(\tau_x > t)$ is ultimately completely monotone. As it was observed in~\cite{bib:ds10}, Remark~4, this already implies formula~\eqref{eq:fptdensityasymp} for all $n \ge 0$, see~\cite{bib:bgt87} (we omit the details).\qed
\end{remark}

\begin{proof}[Proof of Theorem~\ref{th:fptdensityasymp}(b)]
The argument is similar to the proof of part~(a) of the theorem. Again we use the notation of the proof of Lemma~\ref{lem:fptdensityest}. Let $f_{n,t,x}(\lambda)$ be the integrand in~\eqref{eq:fpt}. Fix $t > t_0$. We have
\formula{
 & \sqrt{\psi(1/x^2)} \int_0^\infty f_{n,t,x}(\lambda) d\lambda \\
 & \qquad = \int_0^\infty \frac{\sqrt{\psi(1/x^2)} \, F_\lambda(x)}{\lambda \sqrt{\psi'(\lambda^2)}} \, (\psi(\lambda^2))^{n - 1/2} e^{-t\psi(\lambda^2)} \lambda \psi'(\lambda^2) d\lambda .
}
By~\eqref{eq:flambdaregular}, $\lim_{x \to 0^+} \sqrt{\psi(1/x^2)} \, F_\lambda(x) / (\lambda \sqrt{\psi'(\lambda^2)}) = 1 / \Gamma(1 + \alpha)$. We will use dominated convergence for the integral over an initial interval $(0, B)$, and a simple uniform bound on the remaining interval $[B, \infty)$.

Let $B > 0$. Consider $x$ small enough, so that $\lambda x \le (\pi/2 - \Theta) / 2$ and $\psi(1/x^2) \ge 2 \psi(\lambda^2)$ for $\lambda \in (0, B)$ (recall that $\psi$ is unbounded). By Lemma~\ref{lem:flambdaest}, for $\lambda \in (0, B)$,
\formula{
 \frac{\sqrt{\psi(1/x^2)} \, F_\lambda(x)}{\lambda \sqrt{\psi'(\lambda^2)}} & \le c(\Theta) \, \sqrt{\frac{\lambda^2 x^2 \psi(1/x^2) \psi_\lambda(1/x^2)}{\psi'(\lambda^2) \psi_\lambda(\lambda^2)}} \\
 & = c(\Theta) \, \sqrt{\frac{\psi(1/x^2) (1 - \lambda^2 x^2)}{\psi(1/x^2) - \psi(\lambda^2)}} \\
 & \le c(\Theta) \, \frac{1}{\sqrt{1 - \psi(\lambda^2) / \psi(1/x^2)}} \le \sqrt{2} \, c(\Theta) .
}
Hence, by dominated convergence,
\formula{
 & \lim_{x \to 0^+} \expr{\sqrt{\psi(1/x^2)} \int_0^B f_{n,t,x}(\lambda) d\lambda} \\
 & \qquad = \frac{1}{\Gamma(1 + \alpha)} \int_0^B (\psi(\lambda^2))^{n - 1/2} e^{-t \psi(\lambda^2)} \lambda \psi'(\lambda^2) d\lambda .
}
More precisely, we have
\formula{
 & \lim_{x \to 0^+} \int_0^B \abs{\sqrt{\psi(1/x^2)} f_{n,t,x}(\lambda) - \frac{(\psi(\lambda^2))^{n - 1/2} e^{-t \psi(\lambda^2)} \lambda \psi'(\lambda^2)}{\Gamma(1 + \alpha)}} d\lambda = 0 ,
}
and the convergence is uniform in $t > t_0$, due to monotonicity of of the integrand in $t > t_0$. On the other hand,
\formula{
 \abs{\frac{2}{\pi} \int_B^\infty f_{n,t,x}(\lambda) d\lambda} & \le I_n(t_0, B) ,
}
which converges to $0$ as $B \to \infty$, uniformly in $x$ and $t > t_0$ (by Remark~\ref{rem:remainder}). Hence, by a substitution $z = t \psi(\lambda^2)$,
\formula{
 \hspace*{3em} & \hspace*{-3em} \lim_{x \to 0^+} \expr{\sqrt{\psi(1/x^2)} \int_0^\infty f_{n,t,x}(\lambda) d\lambda} \\
 & = \frac{1}{\Gamma(1 + \alpha)} \int_0^\infty (\psi(\lambda^2))^{n - 1/2} e^{-t\psi(\lambda^2)} \lambda \psi'(\lambda^2) d\lambda \\
 & \hspace{-0em} = \frac{1}{2 \Gamma(1 + \alpha) t^{n + 1/2}} \int_0^\infty z^{n - 1/2} e^{-z} dz = \frac{\Gamma(n + 1/2)}{2 \Gamma(1 + \alpha) t^{n + 1/2}} \/ .
}
Formula~\eqref{eq:fptdensityregular} follows now from~\eqref{eq:fpt}.
\end{proof}

%
%

\section{Examples}
\label{sec:examples}

\subsection{Symmetric stable processes}

These processes, corresponding to $\Psi(\xi) = |\xi|^\alpha$ with $\alpha \in (0, 2]$, have already been studied in Example~6.1 in~\cite{bib:mk10}. In this case $\thet_\lambda = (2 - \alpha) \pi/8$, and Theorem~\ref{th:fpt} reads
\formula[eq:stable1]{
 (-1)^n \frac{d^n}{dt^n} \, \pr_x(\tau_x > t) & = \frac{\sqrt{2 \alpha}}{\pi} \int_0^\infty \lambda^{n \alpha - 1} e^{-t \lambda^\alpha} F(\lambda x) d\lambda , && t , x > 0 ,
}
with (see Example~6.1 in~\cite{bib:mk10})
\formula[eq:stable2]{
\begin{aligned}
 F(\lambda x) & = \sin(\lambda x + (2 - \alpha) \pi/8) - \frac{\sqrt{2 \alpha} \, \sin (\alpha \pi / 2)}{2 \pi} \int_0^\infty \frac{s^\alpha}{1 + s^{2 \alpha} - 2 s^\alpha \cos (\alpha \pi / 2)} \\
 & \hspace*{11em} \times \exp\expr{\frac{1}{\pi} \int_0^\infty \frac{1}{1 + \zeta^2} \, \log \frac{1 - s^2 \zeta^2}{1 - s^\alpha \zeta^\alpha} \, d\zeta} e^{-\lambda s x} ds .
\end{aligned}
}
Note that all above integrands are highly regular functions (for example, $\log ((1 - s^2 \zeta^2) / (1 - s^\alpha \zeta^\alpha))$ is a complete Bernstein function), and thus~\eqref{eq:stable2} is suitable for numerical integration. With a little effort, explicit upper bounds for numerical errors can also be computed. Together with~\eqref{eq:stable1}, this gives \emph{faithful} numerical bounds for $\pr(\tau_x > t)$ and its derivatives in $t$. Plots of cumulative distribution function and density function of $\tau_x$ obtained using this method are given in Figure~\ref{fig:stable}.

\begin{sidewaysfigure}
\centering
\begin{tabular}{c@{}c}
\raisebox{\dimexpr-\height+\ht\strutbox\relax}{\parbox{0.49\textwidth}{\centering\includegraphics[width=0.45\textwidth]{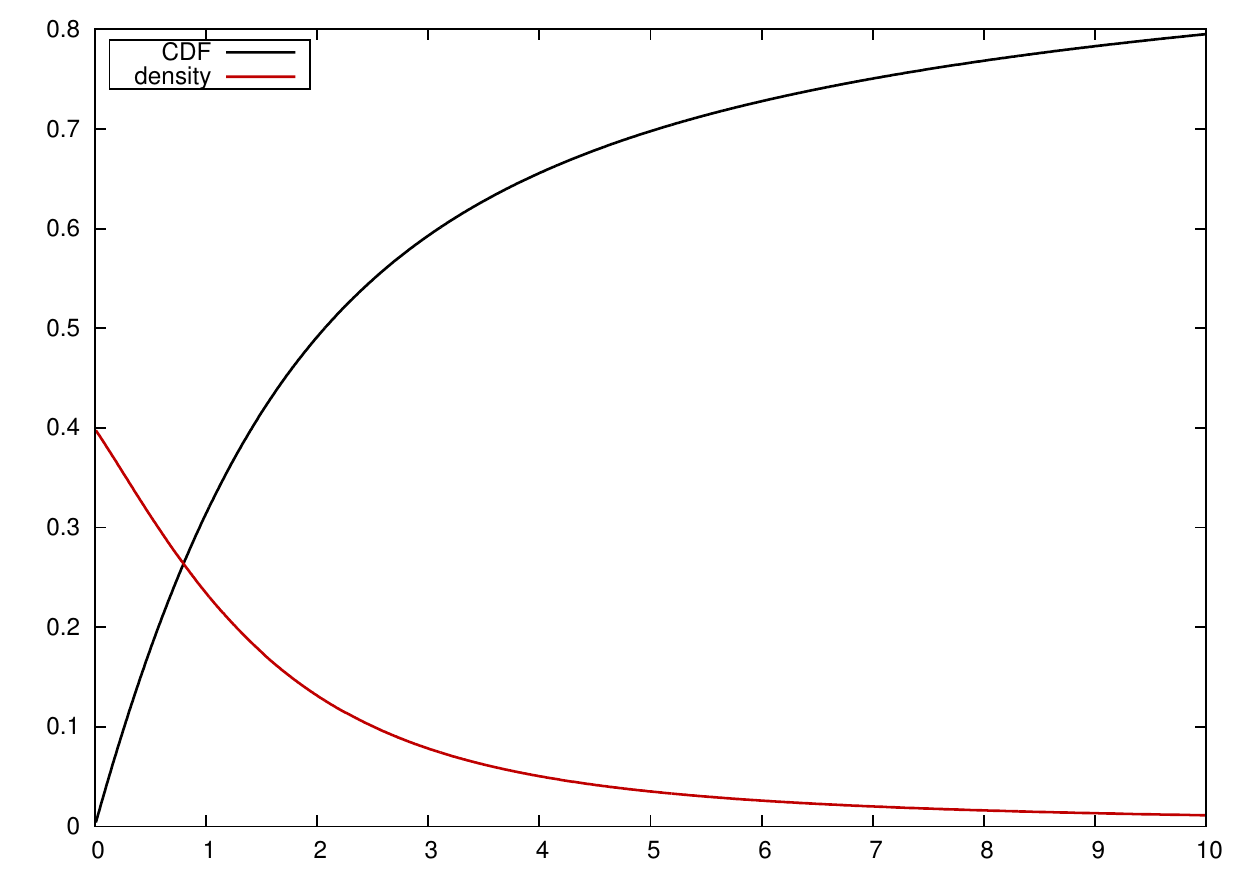}\\{\small (a)}}}&
\raisebox{\dimexpr-\height+\ht\strutbox\relax}{\parbox{0.49\textwidth}{\centering\includegraphics[width=0.45\textwidth]{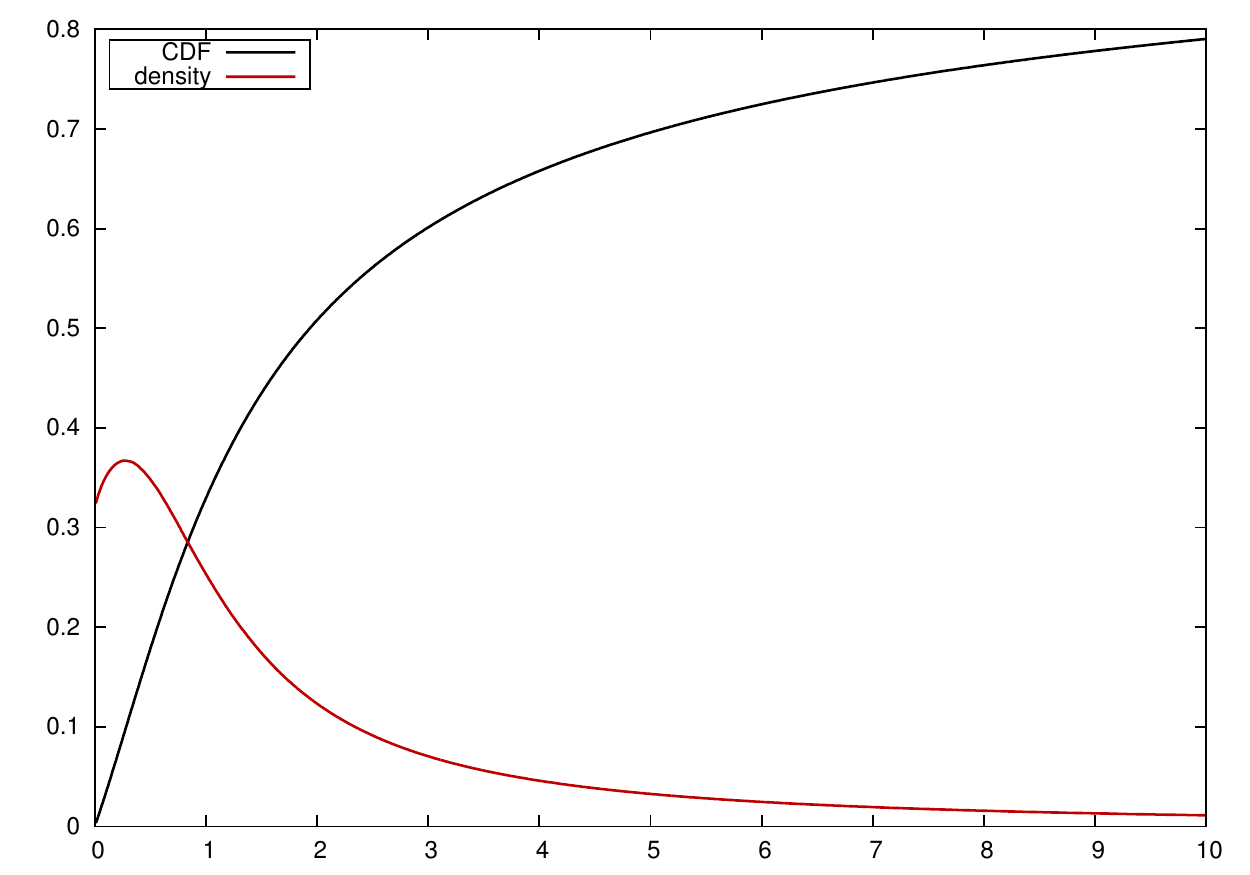}\\{\small (b)}}}\\
\raisebox{\dimexpr-\height+\ht\strutbox\relax}{\parbox{0.49\textwidth}{\centering\includegraphics[width=0.45\textwidth]{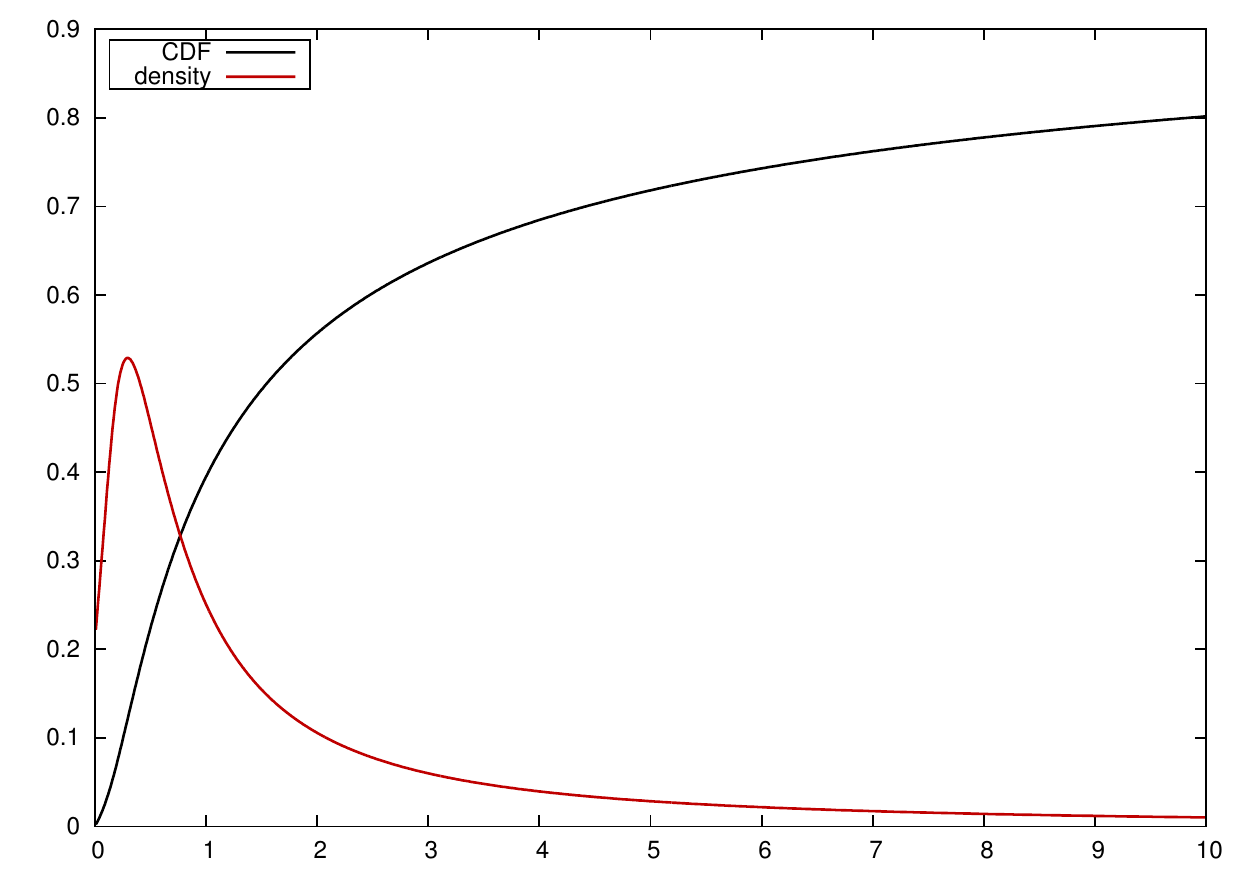}\\{\small (c)}}}&
\raisebox{\dimexpr-\height+\ht\strutbox\relax}{\parbox{0.49\textwidth}{\centering
\caption{\label{fig:stable} Plots of $\pr(\tau_x < t)$ (black) and $\frac{d}{dt} \pr(\tau_x < t)$ (red) for the symmetric $\alpha$-stable L\'evy process, computed with three digits of accuracy, for (a) $\alpha = 0.5$; (b) $\alpha = 1.0$; (c) $\alpha = 1.5$. Calculations are based on Theorems~\ref{th:fpt} and~\ref{th:eigenfunctions} and the following numerical integration scheme. \\[0.1em] Suppose that $f$ is integrated over an interval $I$. First, using \emph{analytical} methods, $f$ is bounded above and below on any subinterval $I'$ by a simpler function (e.g. a polynomial), which is integrated analytically. This yields lower and upper bounds for the integral of $f$ over $I'$. A local adaptive strategy is used to divide $I$ into sub-intervals $I'$ so that the total error on all sub-intervals does not exceed given level. Interval arithmetic is used for nested integrals. \\[0.1em] Plots prepared using \texttt{gnuplot} and a \texttt{C} program with 80-bit precision floating point numbers.}}}
\end{tabular}
\end{sidewaysfigure}

\subsection{Processes with power-type L\'evy-Khintchine exponent}

In this example, we assume that the L\'evy-Khintchine exponent $\Psi(\xi)$ has the form $\Psi(\xi) = \psi(\xi^2)$ for a complete Bernstein function $\psi$, and $\Psi$ is regularly varying of \emph{positive} order $\alpha_0 > 0$ at zero, and of \emph{positive} order $\alpha_\infty > 0$ at $\pm \infty$ (see the first part of Table~\ref{tab:cbf} for some examples). Clearly, in this case $\psi(\xi)$ is regularly varying of orders $\alpha_0/2$ and $\alpha_\infty/2$ at $0$ and $\infty$, respectively. Hence, by Karamata's theory of regularly varying functions (see~\cite{bib:bgt87}), $-\xi \psi''(\xi) / \psi'(\xi)$ converges to $1 - \alpha_0/2$ and $1 - \alpha_\infty/2$ as $\xi \to 0^+$ and $\xi \to \infty$, respectively. 

For simplicity, we assume in addition that the supremum in~\eqref{eq:fpt:a1} is less than one. This condition is satisfied by all examples given in the first part of Table~\ref{tab:cbf}. Note, however, that the condition given in Remark~\ref{rem:fptdensityest}(b) is automatically satisfied, so that our extra assumption can be easily dropped by referring to an improved version of Theorem~\ref{th:fptdensityest}, alluded to in Remark~\ref{rem:fptdensityest}(b).

Theorem~\ref{th:fptdensityest}(b) yields the estimate
\formula{
 \frac{c_1(n)}{t^{n + 1/2} \sqrt{\psi(1 / x^2)}} & \le (-1)^n \, \frac{d^n}{dt^n} \, \pr(\tau_x > t) \le \frac{c_2(n)}{t^{n + 1/2} \sqrt{\psi(1 / x^2)}}
}
for $n \ge 0$, $x > 0$ and $t \ge c_3(n, \alpha) / \psi(1 / x^2)$. Furthermore, by Theorem~\ref{th:fptdensityasymp},
\formula{
 \lim_{x \to 0^+} \expr{\sqrt{\psi(1/x^2)} \, \frac{d^n}{dt^n} \, \pr(\tau_x > t)} & = \frac{(-1)^n \Gamma(n + 1/2)}{\pi \Gamma(1 + \alpha_\infty/2)} \, \frac{1}{t^{n + 1/2}} \/, && t > 0, \, n \ge 0 , \\
 \lim_{t \to \infty} \expr{t^{n + 1/2} \, \frac{d^n}{dt^n} \, \pr(\tau_x > t)} & = \frac{(-1)^n \Gamma(n + 1/2)}{\pi} \, V(x) , && x > 0 , \, n \ge 0.
}

\begin{table}
\centering\small
\begin{tabular}{cccp{12em}p{10em}}
 \hline\hline
 $\Psi(\xi)$ & $\alpha_0$ & $\alpha_\infty$ & \multicolumn{1}{c}{$X_t$} & \multicolumn{1}{c}{restrictions} \\
 \hline
 $\xi^2$ & $2$ & $2$ & Brownian motion & \\
 $|\xi|^\alpha$ & $\alpha$ & $\alpha$ & $\alpha$-stable & $\alpha \in (0,2]$ \\
 $c_1 |\xi|^\alpha + c_2 |\xi|^\beta$ & $\alpha$ & $\beta$ & (sum of two stables) & $\alpha, \beta \in (0,2]$, $\alpha < \beta$ \\
 $(\xi^2+m^{2/\alpha})^{\alpha/2}-m$ & $2$ & $\alpha$ & relativistic $\alpha$-stable & $\alpha \in (0,2)$, $m > 0$ \\
 $((\xi^2+1)^{\alpha/\beta}-1)^{\beta/2}$ & $\beta$ & $\alpha$ & (subordinate relat. stable) & $\alpha, \beta \in (0,2]$, $\alpha < \beta$ \\
 \hline
 $\log(1 + \xi^2)$ & $2$ & $0$ & variance gamma & \\
 $\log(1 + |\xi|^\alpha)$ & $\alpha$ & $0$ & geometric stable & $\alpha \in (0,2]$ \\
 $1 / \log(1 + 1/|\xi|^\alpha)$ & $0$ & $\alpha$ & (not named) & $\alpha \in (0,2]$ \\
 $\log(1 + 1 / \log(1 + 1/\xi^2))$ & $0$ & $0$ & (not named) & \\
 \hline
 $\xi^2 / (1 + \xi^2)$ & $2$ & $0$ & (compound Poisson) & \\
 \hline\hline
\end{tabular}
\bigskip
\caption{Some L\'evy-Khintchine exponents $\Psi(\xi)$, regularly varying both at $0$ (of order $\alpha_0$) and at $\pm\infty$ (of order $\alpha_\infty$). Names of corresponding subordinate Brownian motions are given in column $X_t$. First part of the table contains power-type functions $\Psi$, more singular examples are given in the other parts.}
\label{tab:cbf}
\end{table}

\subsection{Slowly varying L\'evy-Khintchine exponents}
\label{subsec:geomstable}

When the L\'evy-Khintchine exponent $\Psi(\xi)$ has the form $\Psi(\xi) = \psi(\xi^2)$ for a complete Bernstein function $\psi$, and $\Psi$ is regularly varying of order $\alpha_0$ at zero, and of order $\alpha_\infty$ at $\pm \infty$, but at least one of $\alpha_0$, $\alpha_\infty$ is zero, estimates of the distribution of $\tau_x$ become more delicate. In this case Theorem~\ref{th:fptdensityest} cannot be applied, and one needs to refer to either Corollary~\ref{lem:fptdensityest} (which is fairly straightforward, but typically yields sub-optimal results) or the technical Lemma~\ref{lem:fptdensityest}. Note that only $n \ge 1$ need to be considered, as $n = 0$ was studied in general in~\cite{bib:kmr11}.

Suppose that $|\psi''(\xi)| / \psi'(\xi) \le 1 / \xi$ for all $\xi > 0$ (that is, $\ro = 1$ in Theorem~\ref{th:fptdensityest}(b)); this condition is satisfied by all processes in the second part of Table~\ref{tab:cbf}. Then $\Theta = \sup_{\lambda > 0} \thet_\lambda \le \pi / 4$ by Proposition~\ref{prop:thetaest2}. Furthermore, by integration, $\psi'(\xi_1) / \psi'(\xi_2) \le \xi_2 / \xi_1$ when $0 < \xi_1 < \xi_2$, and therefore (cf. the proof of Theorem~\ref{th:fptdensityest}(b))
\formula[eq:geomstable0]{
\begin{aligned}
 I_n(t, \lambda_0) & \le \frac{2}{\pi \sqrt{\lambda_0^2 \psi'(\lambda_0^2)}} \int_{\lambda_0}^\infty e^{-t \psi(\lambda^2)} (\psi(\lambda^2))^{n - 1/2} (2 \lambda \psi'(\lambda^2)) d\lambda \\
 & = \frac{2}{\pi} \, \frac{\Gamma(n + 1/2; t \psi(\lambda_0^2))}{t^{n + 1/2} \sqrt{\lambda_0^2 \psi'(\lambda_0^2)}} \, ,
\end{aligned}
}
where $\Gamma(a; z)$ is the upper incomplete gamma function. However, $\lambda_0^2 \psi'(\lambda_0^2)$ is no longer comparable with $\psi(\lambda_0^2)$. Nevertheless, we can combine~\eqref{eq:geomstable0} with Lemma~\ref{lem:fptdensityest}, to find that
\formula{
\begin{aligned}
 \frac{C_1(n, \pi/4)}{2} \, \frac{1}{t^{n+1/2} \sqrt{\psi(1/x^2)}} & \le (-1)^n \frac{d^n}{dt^n} \, \pr(\tau_x > t) \\
 & \le \expr{C_1(n, \pi/4) + \frac{C_2(n, \pi/4)}{2}} \frac{1}{t^{n+1/2} \sqrt{\psi(1/x^2)}} \, ,
\end{aligned}
}
provided that $t \psi(1/x^2) \ge 1$ (so that the minimum in~\eqref{eq:fptdensityest} is $1 / (t^{n + 1/2} (\psi(1 / x^2))^{1/2})$) and
\formula[eq:geomstable2]{
 \Gamma(n + 1/2; t \psi((\lambda_0(x))^2)) & \le \frac{\pi C_1(n, \pi/4)}{4} \, \sqrt{\frac{(\lambda_0(x))^2 \psi'((\lambda_0(x))^2)}{\psi(1 / x^2)}} \, .
}
Here $C_1(n, \Theta)$, $C_2(n, \Theta)$ are the constants of Lemma~\ref{lem:fptdensityest}, and $\lambda_0(x) = \pi / (8 x)$.

Note that $\Gamma(n + 1/2; z) \le c(n) e^{-z / 2}$ for $z > 0$ for some $c(n)$. Furthermore, $\psi'$ is a decreasing function. After some simplification (we omit the details), this gives the following sufficient condition for~\eqref{eq:geomstable2}:
\formula{
 t \psi(1 / x^2) & \ge c'(n) \expr{1 + \log \expr{1 + \frac{\psi(1 / x^2)}{(1 / x^2) \psi'(1 / x^2)}}}
}
for some $c'(n)$. For convenience, we state this as a separate result.

\begin{proposition}
\label{prop:geomstable}
If $|\psi''(\xi)| / \psi'(\xi) \le 1 / \xi$ for all $\xi > 0$, then there are positive constants $c_1(n)$, $c_2(n)$, $c_3(n)$ such that
\formula[eq:geomstable1]{
 \frac{c_1(n)}{t^{n+1/2} \sqrt{\psi(1/x^2)}} & \le (-1)^n \frac{d^n}{dt^n} \, \pr(\tau_x > t) \le \frac{c_2(n)}{t^{n+1/2} \sqrt{\psi(1/x^2)}}
}
for $n \ge 0$, $x > 0$ and
\formula[eq:geomstable3]{
 t & \ge \frac{c_3(n)}{\psi(1 / x^2)} \expr{1 + \log \expr{1 + \frac{\psi(1 / x^2)}{(1 / x^2) \psi'(1 / x^2)}}} . \qed
}
\end{proposition}

For example, when $\Psi(\xi) = \log(1+|\xi|^\alpha)$ ($\alpha \in (0,2]$), which corresponds to geometric stable processes, condition~\eqref{eq:geomstable3} reads
\formula{
 t & \ge \frac{c(\alpha, n)}{\psi(1 / x^2)} \, (1 + \log(1 + \log(1 + 1 / x))) .
}
In this case by Theorem~\ref{th:fptdensityasymp} we also have
\formula{
 \lim_{x \to 0^+} \expr{\sqrt{\log(1/x)} \, \frac{d^n}{dt^n} \, \pr(\tau_x > t)} & = \frac{(-1)^n \Gamma(n + 1/2)}{\pi \sqrt{\alpha/2}} \, \frac{1}{t^{n + 1/2}} \/ , && t > 0, \, n \ge 0 , \\
 \lim_{t \to \infty} \expr{t^{n + 1/2} \, \frac{d^n}{dt^n} \, \pr(\tau_x > t)} & = \frac{(-1)^n \Gamma(n + 1/2)}{\pi} \, V(x) , && x > 0 , \, n \ge 0.
}

\subsection{Exponents with very slow growth}

Let $\tilde{\psi}(\xi) = \log(1 + \xi^\alpha)$ be the complete Bernstein function studied in the previous example, and let $\psi = \tilde{\psi} \circ ... \circ \tilde{\psi}$ be the $N$-fold composition of $\tilde{\psi}$ ($\alpha \in (0, 2]$, $N \ge 2$). The function $\psi$ is the Laplace exponent of the iterated geometric $\alpha$-stable subordinator, and $\Psi(\xi) = \psi(\xi^2)$ is the L\'evy-Khintchine exponent of the corresponding subordinate Brownian motion. It is easy to verify that~\eqref{eq:fpt:a1} is satisfied. However, \eqref{eq:fpt:a2} holds only if $N = 2$ and $t \ge 1 / \alpha$ ($n$ arbitrary). Hence, for $N = 2$ asymptotic expansion~\eqref{eq:fptdensityregular} in Theorem~\ref{th:fptdensityasymp}(b) is valid only for $t > 1 / \alpha$ (part (a) of the theorem holds for all $x > 0$). When $N \ge 3$, neither part of Theorem~\ref{th:fptdensityasymp} applies.

Note that for any $N \ge 2$ and $\alpha \in (0, 2]$, $\sup_{\xi > 0} (\xi |\psi''(\xi)| / \psi'(\xi)) > 1$, so these examples do not fit into the framework of Theorem~\ref{th:fptdensityest}(b), or even that of Subsection~\ref{subsec:geomstable}. Some estimates of the derivatives of $\pr(\tau_x > t)$ for $N = 2$ and $t > 1 / \alpha$ can be obtained using Lemma~\ref{lem:fptdensityest}, see e.g. Corollary~\ref{cor:fptdensityest}.

\subsection{Compound Poisson process with Laplace distributed jumps}

Let $\Psi(\xi) = \xi^2 / (1 + \xi^2)$. Then the corresponding process $X_t$ is the compound Poisson process with Laplace distributed (i.e. with density function $(1/2) e^{-|x|}$) jumps occuring at unit rate. It was proved in~\cite{bib:kmr11} that
\formula{
 \pr(\tau_x > t) & \approx \min\left(1, \sqrt{\frac{1+x^2}{t}}\right) , && x > 0 , \, t > 0 .
}
Note that~\eqref{eq:fpt:a1} is not satisfied, so the main results of the present article cannot be used. Indeed, we have $\thet_\lambda = \arctan \lambda$ (see~\cite{bib:mk10}, Example~6.6), and so $\sup_{\lambda > 0} \thet_\lambda$ is indeed equal to $\pi/2$.

\subsection{Irregular example}
\label{subsec:irregular}

Estimates for $\thet_\lambda$ are critical for Theorems~\ref{th:fptdensityest} and~\ref{th:fptdensityasymp}. We have $\thet_\lambda \in [0, \pi/2)$, and it is easy to construct examples for which $\lim_{\lambda \to \infty} \thet_\lambda$ is any number in $[0, \pi/2]$. On the other hand, it seems unlikely that $\liminf_{\lambda \to 0^+} \thet_\lambda > \pi/4$ is possible. Below we show a rather irregular example for which $\limsup_{\lambda \to 0^+} \thet_\lambda = \pi/2$ (but $\liminf_{\lambda \to 0^+} \thet_\lambda = 0$).

Consider the complete Bernstein function
\formula{
 \psi(\xi) & = \sum_{k = 1}^\infty p_k \, \frac{\xi}{a_k + \xi} \, ,
}
where, for example, $p_k = 1 / k!$ and $a_k = 1 / (k!)^2$. Fix $q > 0$. We consider $\lambda = (q a_n)^{1/2}$ and let $n \to \infty$. We have
\formula{
 0 \le \psi(q a_n) \expr{p_n \, \frac{q}{1 + q}}^{-1} - 1 & = \sum_{k \ne n} \frac{p_k}{p_n} \, \frac{(1 + q) a_n}{a_k + q a_n} \\
 & \le \sum_{k = 1}^{n-1} \frac{p_k}{p_n} \, \frac{(1 + q) a_n}{a_k} +  \sum_{k = n + 1}^\infty \frac{p_k}{p_n} \, \frac{(1 + q) a_n}{q a_n} \\
 & = (1 + q) \sum_{k = 1}^{n-1} \frac{k!}{n!} + \frac{1 + q}{q} \sum_{k = n + 1}^\infty \frac{n!}{k!} \\
 & \le \frac{1 + q}{n} \sum_{k = 1}^{n-1} \frac{1}{(n - k - 1)!} + \frac{1 + q}{q (n + 1)} \sum_{k = n + 1}^\infty \frac{1}{(k - n - 1)!} \\
 & \le \frac{(1 + q)^2 e}{q n} = O(1/n) .
}
Here and below the constant in the $O(\cdot)$ notation may depend on $q$. In a similar manner,
\formula{
 0 \le \psi'(q a_n) \expr{p_n \, \frac{1}{(1 + q)^2 a_n}}^{-1} - 1 & = \sum_{k \ne n} \frac{p_k}{p_n} \, \frac{(1 + q)^2 a_n a_k}{(a_k + q a_n)^2} \\
 & \le \sum_{k = 1}^{n-1} \frac{p_k}{p_n} \, \frac{(1 + q)^2 a_n}{a_k} +  \sum_{k = n + 1}^\infty \frac{p_k}{p_n} \, \frac{(1 + q)^2}{q} \\
 & \le \frac{(1 + q)^3 e}{q n} = O(1/n) ,
}
and
\formula{
 0 \le |\psi''(q a_n)| \expr{p_n \, \frac{2}{(1 + q)^3 a_n^2}}^{-1} - 1 & = \sum_{k \ne n} \frac{p_k}{p_n} \, \frac{(1 + q)^3 a_n^2 a_k}{(a_k + q a_n)^3} \\
 & \le \sum_{k = 1}^{n-1} \frac{p_k}{p_n} \, \frac{(1 + q)^3 a_n}{a_k} +  \sum_{k = n + 1}^\infty \frac{p_k}{p_n} \, \frac{(1 + q)^3}{q} \\
 & \le \frac{(1 + q)^4 e}{q n} = O(1/n) .
}
It follows that
\formula{
 \psi(q a_n) & = \expr{p_n \, \frac{q}{1 + q}} (1 + O(1/n)) , \\
 \psi'(q a_n) & = \expr{p_n \, \frac{1}{(1 + q)^2 a_n}} (1 + O(1/n)) , \\
 |\psi''(q a_n)| & = \expr{p_n \, \frac{2}{(1 + q)^3 a_n^2}} (1 + O(1/n)) .
}
With the notation of the proof of Proposition~\ref{prop:thetaest1}, for $\lambda^2 = q a_n$ we obtain
\formula{
 a_1^2 & = \frac{\psi(\lambda^2)}{\lambda^2 \psi'(\lambda^2)} = (1 + q) (1 + O(1/n)) , \\ 
 a_2^2 & = \frac{\lambda^2 |\psi''(\lambda^2)|}{2 \psi'(\lambda^2)} = \frac{q}{1 + q} \, (1 + O(1/n)) , \\ 
 a^2 & = a_1^2 - 1 = q (1 + O(1/n)) . 
}
By Proposition~\ref{prop:thetaest1},
\formula{
 \limsup_{n \to \infty} \thet_{(q a_n)^{1/2}} & \le \lim_{n \to \infty} \arctan \sqrt{q (1 + O(1/n))} = \arctan \sqrt{q} ,
}
and, in a similar manner,
\formula{
 \liminf_{n \to \infty} \thet_{(q a_n)^{1/2}} & \ge \frac{1}{\pi} \expr{\expr{\arcsin \sqrt{\frac{q}{1 + q}}}^2 + \expr{\arcsin 1}^2 - \expr{\arcsin \sqrt{\frac{1}{1 + q}}}^2} \\
 & \hspace*{-5em} = \frac{\pi}{4} - \frac{1}{\pi} \expr{\arcsin \sqrt{\frac{1}{1 + q}} + \arcsin \sqrt{\frac{q}{1 + q}}} \expr{\arcsin \sqrt{\frac{1}{1 + q}} - \arcsin \sqrt{\frac{q}{1 + q}}} \\
 & \hspace*{-5em} = \frac{\pi}{4} - \frac{1}{\pi} \, \frac{\pi}{2} \, \arcsin \frac{1 - q}{1 + q} = \frac{1}{2} \, \arccos \frac{1 - q}{1 + q} = \frac{1}{2} \, \arctan \frac{2 \sqrt{q}}{1 - q} = \arctan \sqrt{q} .
}
Hence, $\lim_{n \to \infty} \thet_{(q a_n)^{1/2}} = \arctan q^{1/2}$, and therefore any number in $[0, \pi/2]$ is a partial limit of $\thet_\lambda$ as $\lambda \to 0^+$.

%
%


%
%


\begin{thebibliography}{00}

\bibitem{bib:as72}
M.~Abramowitz, I.~A.~Stegun,
\emph{Handbook of Mathematical Functions with Formulas, Graphs, and Mathematical Tables}.
Dover, New York, 9th edition, 1972.

\bibitem{bib:ak05}
L.~Alili, A.~E.~Kyprianou,
\emph{Some remarks on first passage of L{\'e}vy processes, the American put and pasting principles}.
Ann. Appl. Probab. 15(3) (2005) 2062--2080.

\bibitem{bib:bnmr01}
O.~E.~Barndorff-Nielsen, T.~Mikosch, S.~I.~Resnick (Eds.),
\emph{L{\'e}vy Processes: Theory and Applications}.
Birkh{\"a}user, Boston, 2001.

\bibitem{bib:bd57}
G.~Baxter, M.~D.~Donsker,
\emph{On the distribution of the supremum functional for processes with stationary independent increments}.
Trans. Amer. Math. Soc. 85 (1957) 73--87.

\bibitem{bib:bdp08}
V.~Bernyk, R.~C.~Dalang, G.~Peskir,
\emph{The law of the supremum of a stable L\'{e}vy process with no negative jumps}.
Ann. Probab. 36(5) (2008) 1777--1789.

\bibitem{bib:b96}
J.~Bertoin,
\emph{L\'{e}vy Processes}.
Cambridge Univ. Press, Melbourne, New York, 1996.


\bibitem{bib:b73}
N.~H.~Bingham,
\emph{Maxima of sums of random variables and suprema of stable processes}.
Z.~Wahrscheinlichkeitstheorie Verw. Gebiete 26 (1973) 273--296.

\bibitem{bib:bgt87}
N.~H.~Bingham, C.~M.~Goldie, J.~L.~Teugels,
\emph{Regular Variation}.
Cambridge University Press, Cambridge, 1987.

\bibitem{bib:bbkrsv09}
K.~Bogdan, T.~Byczkowski, T.~Kulczycki, M.~Ryznar, R.~Song, Z.~Vondra\v{c}ek,
\emph{Potential Analysis of Stable Processes and its Extensions}.
Lecture Notes in Mathematics 1980, Springer, 2009.


\bibitem{bib:d56}
D.~A.~Darling,
\emph{The maximum of sums of stable random variables}.
Trans. Amer. Math. Soc. 83 (1956) 164--169.

\bibitem{bib:d87}
R.~A.~Doney,
\emph{On Wiener-Hopf factorisation and the distribution of extrema for certain stable processes}.
Ann. Probab. 15(4) (1987) 1352--1362.

\bibitem{bib:d07}
R.~A.~Doney,
\emph{Fluctuation Theory for L{\'e}vy Processes}.
Lecture Notes in Math. 1897, Springer, Berlin, 2007.

\bibitem{bib:dr11}
R.~A.~Doney, V.~Rivero,
\emph{Asymptotic behaviour of first passage time distributions for L\'evy processes}/
Preprint, 2011, arXiv:1107.4415v1.

\bibitem{bib:ds10}
R.~A.~Doney, M.~S.~Savov,
\emph{The asymptotic behavior of densities related to the supremum of a stable process}.
Ann. Probab. 38(1) (2010) 316--326.


\bibitem{bib:gj09}
P.~Graczyk, T.~Jakubowski,
\emph{On exit time of symmetric $\alpha$-stable processes}.
Stoch. Proc. Appl. 122 (2012) 31--41.

\bibitem{bib:gj10}
P.~Graczyk, T.~Jakubowski,
\emph{On Wiener-Hopf factors of stable processes}.
Ann. Inst. Henri Poincar{\'e} (B) 47(1) (2010) 9--19.

\bibitem{bib:gn86}
P.~E.~Greenwood, A.~A.~Novikov,
\emph{One-sided boundary crossing for processes with independent increments}.
Teor. Veroyatnost. i Primenen. 31(2) (1986) 266--277.


\bibitem{bib:hk11}
F.~Hubalek, A.~Kuznetsov,
\emph{A convergent series representation for the density of the supremum of a stable process}.
Elect. Comm. Probab. 16 (2011) 84--95.

\bibitem{bib:hk09}
T.~R.~Hurd, A.~Kuznetsov,
\emph{On the first passage time for Brownian motion subordinated by a L\'{e}vy process}.
J. Appl. Prob. 46 (2009) 181--198.

\bibitem{bib:ka08}
E.~Katzav, M.~Adda-Bedia,
\emph{The spectrum of the fractional Laplacian and First-Passage-Time statistics}.
EPL 83, 30006 (2008).

\bibitem{bib:ksv9}
P.~Kim, R. Song, Z. Vondra\v{c}ek,
\emph{Boundary Harnack principle for subordinate Brownian motions}.
Stoch. Proc. Appl. 119 (2009) 1601--1631.

\bibitem{bib:ksv10a}
P.~Kim, R. Song, Z. Vondra\v{c}ek,
\emph{On the potential theory of one-dimensional subordinate Brownian motions with continuous components}.
Potential Anal. 33 (2010) 153--173.

\bibitem{bib:ksv10}
P.~Kim, R.~Song, Z.~Vondra\v{c}ek,
\emph{Two-sided Green function estimates for killed subordinate Brownian motions}.
To appear in Proc. London Math. Soc., arXiv:1007.5455v2.

\bibitem{bib:ksv11}
P.~Kim, R.~Song, Z.~Vondra\v{c}ek,
\emph{Potential theory of subordinate Brownian motions revisited}.
To appear in a volume in Honor of Prof. Jiaan Yan, arXiv:1102.1369v2.

\bibitem{bib:kkm07}
T.~Koren, J.~Klafter, M.~Magdziarz,
\emph{First passage times of L\'{e}vy flights coexisting with subdiffusion}.
Phys. Rev.~E 76 (2007) 031129.


\bibitem{bib:k10a}
A.~Kuznetsov,
\emph{Wiener-Hopf factorization and distribution of extrema for a family of L\'{e}vy processes}.
Ann. Appl. Prob. 20(5) (2010) 1801--1830.

\bibitem{bib:k10}
A.~Kuznetsov,
\emph{On extrema of stable processes}.
Ann. Probab. 39(3) (2011) 1027--1060.


\bibitem{bib:mk10}
M.~Kwa{\'s}nicki,
\emph{Spectral analysis of subordinate Brownian motions in half-line}.
Studia Math. 206(3) (2011) 211--271.

\bibitem{bib:kmr11}
M.~Kwa{\'s}nicki, J.~Ma{\l}ecki, M.~Ryznar,
\emph{Suprema of L{\'e}vy processes}.
To appear in Ann. Probab., arXiv:1103.0935v1.

\bibitem{bib:k06}
A.~E.~Kyprianou,
\emph{Introductory Lectures on Fluctuations of L\'{e}vy Processes with Applications}.
Universitext, Springer-Verlag, Berlin, 2006.

\bibitem{bib:m73}
Z.~A.~Melzak,
\emph{Companion to Concrete Mathematics}.
John Wiley \& Sons, New York, 1973.

\bibitem{bib:p59}
R.~Pyke,
\emph{The Supremum and Infimum of the Poisson Process}.
Ann. Math. Stat. 30(2) (1959) 568--576.

\bibitem{bib:s99}
K.~Sato,
\emph{L{\'e}vy Processes and Infinitely Divisible Distributions}.
Cambridge Univ. Press, Cambridge, 1999.

\bibitem{bib:ssv10}
R.~Schilling, R.~Song, Z.~Vondra\v{c}ek,
\emph{Bernstein Functions: Theory and Applications}.
De Gruyter, Studies in Math. 37, Berlin, 2010.

\bibitem{bib:s80}
M.~L.~Silverstein,
\emph{Classification of coharmonic and coinvariant functions for a L{\'e}vy process}.
Ann. Probab. 8(3) (1980) 539--575.

\bibitem{bib:z64}
V.~M.~Zolotarev,
\emph{The first-passage time of a level and the behaviour at infinity for a class of processes with independent increments}.
Theor. Probab, Appl. 9 (1964) 653--664.

\end{thebibliography}
\end{document}